\documentclass[reqno]{amsart}
\usepackage{amssymb}
\usepackage{amsmath}
\usepackage{enumerate}
\usepackage{srcltx}
\usepackage[mathscr]{euscript}
\usepackage{times}

\makeatletter
\@addtoreset{equation}{section}
\makeatother

\newtheorem{theorem}{Theorem}[section]
\newtheorem{lemma}[theorem]{Lemma}
\newtheorem{proposition}[theorem]{Proposition}
\newtheorem{corollary}[theorem]{Corollary}

\newtheorem{definition}[theorem]{Definition}

\newcounter{as}

\newcommand{\scs}[1]{{\scriptscriptstyle #1}}
\newcommand{\mc}[1]{{\mathcal #1}}
\newcommand{\mf}[1]{{\mathfrak #1}}
\newcommand{\mb}[1]{{\mathbf #1}}
\newcommand{\bb}[1]{{\mathbb #1}}
\newcommand{\bs}[1]{{\boldsymbol #1}}
\newcommand{\ms}[1]{{\mathscr #1}}

\newcommand{\<}{\langle}
\renewcommand{\>}{\rangle}

\title[A martingale problem for an absorbed diffusion]
{A martingale problem for an absorbed diffusion: the nucleation
  phase of condensing zero range processes }
  
\author{J. Beltr\'an, M. Jara, C. Landim}

\address{\noindent IMCA, Calle los Bi\'ologos 245, Urb. San C\'esar
  Primera Etapa, Lima 12, Per\'u and PUCP, Av. Universitaria 1801, San
  Miguel, Lima 32, Per\'u.  \newline e-mail: \rm
  \texttt{johel.beltran@pucp.edu.pe} }

\address{\noindent IMPA, Estrada Dona Castorina 110, CEP 22460 Rio de
  Janeiro, Brasil.  \newline e-mail: \rm \texttt{mjara@impa.br} }

\address{\noindent IMPA, Estrada Dona Castorina 110, CEP 22460 Rio de
  Janeiro, Brasil and CNRS UMR 6085, Universit\'e de Rouen, Avenue de
  l'Universit\'e, BP.12, Technop\^ole du Madril\-let, F76801
  Saint-\'Etienne-du-Rouvray, France.  \newline e-mail: \rm
  \texttt{landim@impa.br} }

\begin{document}


\begin{abstract}
  We prove uniqueness of a martingale problem with boundary conditions
  on a simplex associated to a differential operator with an unbounded
  drift.  We show that the solution of the martingale problem remains
  absorbed at the boundary once it attains it, and that, after hitting
  the boundary, it performs a diffusion on a lower dimensional
  simplex, similar to the original one. We also prove that in the
  diffusive time scale condensing zero-range processes evolve as this
  absorbed diffusion.
\end{abstract}

\maketitle

\section{Introduction}
\label{rec04}

It has been observed, in several different contexts, that some
zero-range processes whose jump rates decrease to a positive constant
exhibit condensation: in the stationary state, above a certain
critical density, a macroscopic fraction of the particles concentrate
on a single site \cite{jmp, gss, eh, fls, al, al2, bl3, agl}. 

We investigated in \cite{bl3, l3} the evolution of the condensate in
the case where the total number of sites remains fixed while the total
number of particles diverges.  We proved that in an appropriate time
scale the condensate -- the site where all but a negligible fraction
of particles sit -- evolves according to a Markov chain whose
jump rates are proportional to the capacities of the underlying random
walks performed by the particles in the zero-range process.

We examine in this article how the condensate is formed in the case
where the set of sites, denoted by $S$, remains fixed while the total
number of particles, $N$, tends to infinity. Consider an initial
configuration in which each site is occupied by a positive fraction of
particles. Since in the stationary state almost all particles occupy
the same site, as time evolves we expect to observe a progressive
concentration of particles on a single site.

\smallskip\noindent{\bf Absorbed diffusions}.
To describe the asymptotic dynamics we were led to analyze a diffusion
with boundary conditions whose drift diverges as the process
approaches the boundary and which remains glued to the boundary once
it attains it.  More precisely, denote by $\Sigma$ the simplex $\{x\in
\bb R^S : x_j\ge 0 \,,\, \sum_j x_j =1\}$, where $x_j$ represents the
fraction of particles at site $j\in S$. Far from the boundary of
$\Sigma$ the process evolves as a standard diffusion with bounded and
smooth coefficients, whereas close to the boundary $\{x\in \Sigma :
x_j =0\}$ the drift becomes proportional to $b/x_j$, where $b>1$ is a
fixed parameter, while the variance remains bounded by $1$. In
consequence, when the process approaches the boundary of $\Sigma$ it
is strongly driven to it. Once the boundary $\{x\in \Sigma :
\sum_{j\in A} x_j =0\}$, $A\subset S$, is attained, the process
remains absorbed at this boundary, where it performs a new diffusion,
similar to the original one, but in a lower dimensional space. This
mechanism is iterated and the dimension of the space in which the
diffusion occurs decreases progressively until all but one coordinate
vanish. At this point the process remains trapped in this
configuration for ever.

We called these dynamics ``absorbed'' diffusions to distinguish them
from ``sticky'' diffusions \cite{iw} which bounces at the boundary,
but which have a positive local time at the boundary. 

We did not find in the literature examples of diffusions whose
absorption at the boundary arises from a divergence of the drift close
the the boundary. Diffusions which are absorbed at the boundary due to
a singularity of the covariance matrix of Wright-Fisher type have been
examined in \cite{fk1}.

\smallskip\noindent{\bf Hydrodynamic limit}.  The most popular methods
to derive the hydrodynamic equations of the conserved quantity of an
interacting particle systems relies on the so-called one and two block
estimates \cite{kl1}. The condensing zero-range processes examined in
this article form a class of dynamics for which the one and two blocks
estimate in their classical form do not hold, precisely because of the
condensation of particles.

We examine in this article how particles accumulate on a single site
in the diffusive time scale when the total number of sites is fixed
and the the total number of particles diverges. This is called the
nucleation phase of the condensing zero-range dynamics. A most
interesting open problem is the behavior of the same model, when the
number of sites increases together with the number of particles, where
a convergence towards a self-similar distribution is expected, or the
proof of the hydrodynamical limit of the model when the initial
density profile is super-critical. A first step in this direction has
been performed in \cite{sta1}, where an alternative version of the one
block estimated is proved together with the hydrodynamical limit for
initial profiles bounded below by the critical density, a situation in
which there is no condensation in the diffusive time
scale. 

\smallskip\noindent{\bf Results}.
The first main result of this article asserts that there exists a
unique solution to the martingale problem associated to a second-order
differential operator of an absorbed diffusion.  The second main
result states that in the diffusive time scale $N^2$ the fraction of
particles of condensing zero-range processes on finite sets evolves as
an absorbed diffusion.

We faced two main obstacles in this article.  The first consisted in
the proof that the solution of the martingale problem remains absorbed
at the boundary once it attains it, and that, after hitting the
boundary, the solution performs a diffusion on a lower dimensional
space, similar to the original one. These results and Stroock and
Varadhan \cite{sv71, sv79} theory, with some slight modifications due
to the unboundedness of the drift, yield uniqueness of the martingale
problem. The second main difficulty consisted in the proof of the
tightness of the condensing zero-range processes in the diffusive time
scale, which required a replacement lemma.

\section{Notation and Results}
\label{rec02}

We present in this section the two main results of the article.

\smallskip\noindent{\bf \ref{rec02}.1. The underlying Markov chain.}
Fix a finite set $S=\{1, \dots, L\}$, and consider an irreducible,
continuous-time Markov chain $(x_t)_{t\ge 0}$ on $S$.  Denote by $\bs
r = \{r(j,k) : j, k\in S\}$ the jump rates, so that the generator $\mc
L$ of this Markov chain is
\begin{equation*}
(\mc L f) (j)\;=\; \sum_{k\in S} r(j,k) \{f(k) - f(j)\} \;.
\end{equation*}
Assume, without loss of generality, that $r(j,j)=0$ for all $j\in S$,
and denote by $\lambda(j)$ the holding rate at $j$, $\lambda(j) =
\sum_{k\in S} r(j,k)$. Let $\bs m =\{ m_j : j\in S\}$ be an invariant
measure for $\bs r$, and let $M_j = m_j \lambda(j)$, $j\in S$, so that
$M_j$ is an invariant measure for the embedded discrete-time
chain. Note that we do not assume $\bs m$ to be a probability measure
nor reversible for $\bs r$.

\smallskip\noindent{\bf \ref{rec02}.2. Condensing zero-range processes.} 
Denote by $\eta$, $\xi$ the elements of $\bb N^{S}$, $\bb
N:=\{0,1,2,\dots\}$, so that $\eta(j)$, $j\in S$, represents the
number of particles at site $j$ for the configuration $\eta$. Denote
by $E_{\scs N}$, $N\ge 1$, the set of configurations with $N$
particles:
\begin{equation*}
E_{\scs N}\;:=\;\big\{\eta\in {\bb N}^{S} : \sum_{j\in S} \eta(j) = N \big\}\;. 
\end{equation*}

Fix $b>1$. For each $j\in S$, consider a jump rate $g_j : \bb N\to\bb
R_+$ such that $g_j(0)=0$ and
\begin{equation}
\label{propg}
\lim_{n\to \infty} n\left(\frac{g_j(n)}{m_j}-1 \right) \;=\; b\;.
\end{equation}
In particular, $g_j(n)\to m_j$ and there exists $n_0(j)\in \bb N$ such
that $g_j(n)>m_j$ for all $n\ge n_0(j)$. Hence, roughly, each function
$g_j$ decreases and converges to $m_j$ at rate \eqref{propg}.

Denote by $\eta^{\scs N}=(\eta^{\scs N}_t)_{t\ge 0}$
the zero range process on $S$ associated to the jump rates $r(j,k)$
and $g_j$. This is the continuous-time Markov chain on $\bb N^S$ in
which a particle jumps from $j$ to $k$ at rate $g_j(\eta(j)) r(j,k)$.
The generator $L_{\scs N}$ of this chain acts on functions $F: E_{\scs
  N}\to\bb R$ as
\begin{equation}
\label{f16}
(L_{\scs N} F) (\eta) \;=\; \sum_{j,k\in S}
g_j(\eta(j)) \, r(j,k) \, \big[ F(\sigma^{j,k}\eta) - F(\eta) \big] \;,
\end{equation}
where $\sigma^{j,k}\eta$ stands for the configuration obtained from
$\eta$ by moving a particle from $j$ to $k$: 
\begin{equation*}
(\sigma^{j,k}\eta)(\ell)\;=\;\left\{
\begin{array}{ll}
\eta(j)-1, & \textrm{for $\ell=j$} \\
\eta(k)+1, & \textrm{for $\ell=k$} \\
\eta(\ell), & \rm{otherwise}\;. \\
\end{array}
\right.  
\end{equation*}

Denote by $\Sigma$ the simplex
\begin{equation*}
\Sigma \;=\; \big\{ x \in \bb R^S_+ : \sum_{j\in S} x_j =1 \big\}\;,
\end{equation*}
and by $\Sigma_N$ the $N$-discretization of $\Sigma$:
\begin{equation*}
\Sigma_N \;=\; \big\{ x \in \Sigma : N x_j\in\bb N \,,\, j\in S \,
\big\}\;, \quad N\ge 1 \;.
\end{equation*}

Fix a configuration $\eta\in E_{\scs N}$ and let $\eta^N$ be started
at $\eta$. We consider the rescaled process
\begin{equation*}
X^{\scs N}_t \;:=\; \frac 1N \big(\eta^{\scs N}_{tN^2}(1) ,
\dots, \eta^{\scs N}_{tN^2}(L) \big)\;, \quad t\ge 0\;.
\end{equation*}
Clearly, $(X^{\scs N}_t)_{t\ge 0}$ is a $\Sigma_N$-valued Markov
chain whose generator $\mf L_{\scs N}$ acts on functions
$F:\Sigma_N\to \bb R$ as
\begin{equation}
\label{tay1}
(\mf L_{\scs N} F) (x) \;=\; N^2\sum_{j,k\in S}
g_j(N x_j) \, r(j,k) \, \big[ F(x + \frac{\bs e_k - \bs e_j}{N}) -
F(x) \big] \;,
\end{equation}
where $\{\bs e_i : i\in S\}$ represents the canonical basis of
$\bb R^S$.

\smallskip\noindent{\bf \ref{rec02}.3. Martingale problem.}  One of
the main result of this article states that the Markov chain $X^N_t$
converges in law to a diffusion on $\Sigma$. To introduce the
generator of the diffusion we first define its domain. Denote by
$C^n(\Sigma)$, $n\ge 1$ the set of functions $F:\Sigma \to \bb R$
which are $n$-times continuously differentiable. We let
$\partial_{x_k}F$ and $\nabla F$ stand for partial derivative with
respect to the variable $x_k$ and for the gradient, respectively, of
$F\in C^1(\Sigma)$.

Let $\{\bs v_j \in \bb R^S : j\in S\}$ be the vectors 
\begin{equation*}
\bs v_j\;:=\;\sum_{k\in S} r(j,k)\, \{\bs e_k - \bs e_j\}\;,\quad j\in S\;,
\end{equation*} 
and define the vector field ${\bf b}: \Sigma \to \bb R^S$ by
\begin{equation}
\label{03}
{\bf b}(x) \;:=\; b \sum_{j\in S} {\bf 1} \{x_j\not= 0\} \, 
\frac{m_j}{x_j}  \, \bs v_j \;.
\end{equation}

\begin{definition}
\label{bc1}
For each $j\in S$, let $\mc D_j$, be the space of functions $H$ in
$C^2(\Sigma)$ for which the map $x\mapsto [\bs v_j \cdot \nabla
H(x)]/x_j \, \mb 1\{x_j>0\}$ is continuous on $\Sigma$, and let $\mc
D_A:=\cap_{j\in A}\mc D_j$, $\varnothing \subsetneq A\subseteq S$.
\end{definition}

Clearly, if $H$ belongs to $\mc D_j$, $\bs v_j\cdot\nabla H (x)=0$ for
any $x\in \Sigma$ such that $x_j=0$. Moreover, if $H$ belongs to $\mc
D_S$, $x\mapsto {\bf b}(x)\cdot \nabla H(x)$ is continuous on
$\Sigma$. Finally, to prove that a function $H$ in $C^2(\Sigma)$
belongs to $\mc D_j$, we need to show that for every $z\in \Sigma$
such that $z_j=0$,
\begin{equation}
\label{cf02}
\lim_{\substack{x\to z\\x_j>0}} \frac{\bs v_j \cdot \nabla H(x)}{x_j}
\;=\; 0\;.
\end{equation}
In this formula, the limit is carried over points $x$ in $\Sigma$
which converge to $z$ and such that $x_j>0$. 

Let us now introduce the generator for the limiting diffusion. Let
$\mf L:C^2(\Sigma)\to\bb R$ be the second order differential operator
given by
\begin{equation}
\label{gens}
(\mf L F)(x) \;:=\; {\bf b} (x) \cdot \nabla F(x)  
\;+\; \frac{1}{2} \sum_{j,k\in S} m_j\, r(j,k) \,
(\partial_{x_k} - \partial_{x_j})^2 F (x)\;,
\end{equation}
for any $F\in C^2(\Sigma)$. Thus, operator $\mf L$ depends on three
parameters: the coefficient $b>1$, the jump rates $\bs r$ and the
measure $\bs m$. Of course, for $H\in \mc D_S$, the function $ \mf
LH:\Sigma \to \bb R$ is continuous.

Let $\langle \cdot, \cdot \rangle_{\bs m}$ be the $L^2$-inner product
with respect to the measure $\bs m$. Denote by ${\bf a} =
(a_{i,j})_{i,j\in S}$ the matrix whose entry $(i,j)$ is given by
$a_{i,j} = \langle {\bs e}_i , - \mc L {\bs e}_j\rangle_{\bs m}$ and
denote by ${\bf a}_s$ the symmetric matrix ${\bf a}_{\rm s}=(1/2)({\bf
  a} + {\bf a}^{\rm t})$ where ${\bf a}^{\rm t}$ stands for the
transpose of $\bf a$. So ${\bf a}_{\rm s}$ is the matrix corresponding
to the Dirichlet form of the symmetric part of $(\mc L,{\bs m})$. With
this notation we may write the generator $\mf L$ as
\begin{equation}
\label{f11}
(\mf L F)(x) \;=\; {\bf b} (x) \cdot \nabla F(x)  
\;+\; {\rm Tr }\,  [ {\bf a}_{\rm s} \times {\rm Hess} \, F (x)] \;,
\end{equation}
where ${\rm Tr } \, \bb M$ stands for the trace of the matrix $\bb M$ and
${\rm Hess } \, F$ for the Hessian of $F$.

We now characterize the limiting diffusion as the solution of the
martingale problem corresponding to $(\mf L,\mc D_S)$. Denote by
$C(\bb R_+, \Sigma)$ the space of continuous trajectories $\omega:\bb
R_+\to \Sigma$ endowed with the topology of uniform convergence on
bounded intervals. Every probability measure on $C(\bb R_+, \Sigma)$
will be defined on the corresponding Borel $\sigma$-field $\ms F$. We
denote by $X_t:C(\bb R_+, \Sigma)\to \Sigma$, $t\ge 0$ the process of
coordinate maps and by $(\ms F_t)_{t\ge 0}$ the generated filtration
$\ms F_t:=\sigma(X_s:s\le t)$, $t\ge 0$. A probability measure $\bb P$
on $C(\bb R_+, \Sigma)$ is said to start at $x\in \Sigma$ when $\bb
P[X_0=x]=1$. In addition, we shall say that $\bb P$ is a solution
for the $\mf L$-martingale problem if, for any $H\in \mc D_S$,
\begin{equation}
\label{f06}
H(X_{t}) - \int_{0}^{t} (\mf L H)(X_s) \, ds\;,\quad t\ge 0
\end{equation}
is a $\bb P$-martingale with respect to the filtration $(\ms F_t)_{t\ge 0}$.

\begin{theorem}
\label{r03}
For each $x\in \Sigma$, there exists a unique probability measure on
$C(\bb R_+,\Sigma)$, denoted by $\bb P_x$, which starts at $x$ and is
a solution of the $\mf L$-martingale problem.
\end{theorem}

The uniqueness stated in this theorem will be proved in Section
\ref{rec03}. The existence is established in Section
\ref{sec02}. Furthermore, we shall prove in Subsection \ref{adpro}
that $\{\bb P_x : x\in \Sigma\}$ is actually Feller continuous and
defines a strong Markov process.

\smallskip\noindent{\bf \ref{rec02}.4. An absorbed diffusion.}
Before proceeding, we give a more precise description of the typical
path under $\{ \bb P_x : x\in \Sigma\}$. We start introducing the
absorbing property. For each $x\in \Sigma$, denote
\begin{equation*}
\ms A(x) \;:=\; \{j\in S : x_j =0\} \quad \rm{and} \quad \ms B(x)\;:=\; \ms A(x)^c\;.
\end{equation*}
For all nonempty subset $B\subseteq S$ define $h_B$ as the first time
one of the coordinates in $B$ vanishes
\begin{equation}
\label{hb1}
h_B \;:=\; \inf\{ t \ge 0 : \prod_{j\in B} X_t(j) =0 \} \;.
\end{equation}
Let $(\theta_t)_{t\ge 0}$ stands for the semigroup of time translation
in $C(\bb R_+, \Sigma)$. Now define recursively the sequence
$(\sigma_n,\ms B_n)_{n\ge 0}$ as follows. Set $\sigma_0:=0$, $\ms
B_0:=\ms B(X_0)$. For $n\ge 1$, we define
\begin{equation}
\label{sib}
\sigma_{n} \;:=\; \sigma_{n-1} + h_{\ms B_{n-1}}\circ
\theta_{\sigma_{n-1}}\;, 
\quad  \ms B_{n} \;:=\; \{ j\in S : X_{\sigma_n}(j)>0\}
\end{equation}
on $\{\sigma_{n-1} < \infty\}$ and $\sigma_n := \infty$, $\ms B_{n} :=
\ms B_{n-1}$ on $\{\sigma_{n-1} = \infty\}$. We also denote
\begin{equation*}
\ms A_n \;:=\; \ms B_{n}^c \;,\quad n\ge 0 \;.
\end{equation*}
We shall say that a probability $\bb P$ on $C(\bb R_+,\Sigma)$ is
\emph{absorbing} if
\begin{equation*}
\bb P\{ \textrm{$\ms A_n \subseteq \ms A(X_{t})$ 
for all $t\ge \sigma_n$} \} \;=\; 1\;, \quad \textrm{for every $n\ge 0$}\;.
\end{equation*}
Clearly, if $\bb P$ is absorbing then, $\bb P$-a.s., the sequence of
subsets $(\ms A_n)_{n\ge 0}$ is increasing and
\begin{equation*}
\exists\; 1\le n_0 \le |\ms B_0| \quad \text{such that}\quad
\sigma_{n_0} = \infty \quad \textrm{and} \quad \textrm{$\ms A_{n-1}
\subsetneq \ms A_{n}$\,,\; for $1\le n < n_0$}\;.
\end{equation*}
In particular, observe that if $\bb P$ is absorbing and starts at $\bs
e_j$ for some $j\in S$, then
\begin{equation}
\label{trap}
\bb P [X_t= \bs e_j, \forall t\ge 0] \;=\; 1 \;.
\end{equation}

In Section \ref{rec03} we prove the following result

\begin{theorem}\label{amp0}
For each $x\in \Sigma$, the probability $\bb P_x$ is absorbing.
\end{theorem}
Furthermore, we shall prove in Proposition \ref{r06} that 
\begin{equation}
\label{fsig}
{\rm if} \quad x\not\in\{\bs e_j:j\in S\} 
\quad\textrm{then}\quad  \bb E_x [\sigma_1] \; < \; \infty \;.
\end{equation}

\smallskip\noindent{\bf \ref{rec02}.5. Behavior after absorption.}  In
order to describe more precisely the evolution of the diffusion
process after being absorbed at a boundary, for each $B\subseteq S$
with at least two elements, consider the simplex
\begin{equation*}
\Sigma_B \;:=\; \{ x\in \bb R_+^B : \sum_{j\in B}x_j = 1 \}\;,
\end{equation*}
and the space $C^2(\Sigma_B)$ of functions $f:\Sigma_B\to \bb R$ which
are twice-continuously differentiable.  Denote by $\bs r^B=\{
r^B(j,k): j,k\in B\}$ the jump rates of the trace of the Markov chain
$x_t$ on $B$. The definition of the trace of a Markov chain is
recalled in Section \ref{sec01}. 

Let $\{\bs v^B_j : j\in B\}$ be the vectors in $\bb R^B$ defined by
\begin{equation}
\label{vbj}
\bs v^B_j\;:=\;\sum_{k\in B} r^B(j,k)\, \{\bs e_k - \bs e_j\}\;,
\end{equation}
where $\{{\bs e}_j : j\in B\}$ stands for the canonical basis of $\bb
R^B$, and let ${\bf b}^B: \Sigma_B \to \bb R^B$ be the vector field
defined by
\begin{equation*}
{\bf b}^B(x) \;:=\; b \sum_{j\in B} 
\frac{m_j}{x_j}  \, \bs v^B_j {\mb 1}\{x_j>0\}\,,\quad x\in
\Sigma_B\;. 
\end{equation*}

Denote by $\mc D_{B,A}$, $\varnothing \subsetneq A \subseteq B$, the
space of functions $H$ in $C^2(\Sigma_B)$ for which the map $x\mapsto
[\bs v^B_j \cdot \nabla H(x)]/x_j$ is continuous on $\Sigma_B$ for
$j\in A$, and let $\mf L_B$ be the operator which acts on functions in
$C^2(\Sigma_B)$ as in equation \eqref{gens}, but with the parameter
$\bs r$ replaced by $\bs r^B$,
\begin{equation}
\label{lb1}
(\mf L_B f)(x) \;:=\; {\bf b}^B (x) \cdot \nabla f(x)  
\;+\; \frac{1}{2} \sum_{j,k\in B} m_j\, r^B(j,k) \,
(\partial_{x_k} - \partial_{x_j})^2 f(x)\;,
\end{equation}
for $x\in \Sigma_B$, 

Fix $x$ in $\Sigma$ and assume that $\ms A(x)=\{j\in S : x_j=0\} \not
= \varnothing$. By Theorem \ref{amp0}, $\bb P_x$ is concentrated on
trajectories $X_t$ which belong to $C(\bb R_+, \Sigma_B)$, where $B=
\ms A(x)^c$. Denote by $\bb P^{B}_x$ the measure $\bb P_x$
restricted to $C(\bb R_+, \Sigma_B)$: 
\begin{equation*}
\bb P^{B}_x[ \bs A]\;=\; \bb P_x[ \bs A]\;, \quad
\bs A\subset C(\bb R_+, \Sigma_B)\;,
\end{equation*}
which is a probability measure on $C(\bb R_+, \Sigma_B)$.  For each
$x\in \Sigma$ and nonempty $S_0\subseteq S$, denote by $x_{S_0}$ the
coordinates of $x$ in $S_0$.

\begin{proposition}
\label{sf01}
Fix $x$ in $\Sigma$ and assume that $\ms A(x)=\{j\in S : x_j=0\} \not
= \varnothing$. Let $B=\ms A(x)^c$. The measure $\bb P^{B}_x$ starts
at $x_B$ and solves the $\mf L_B$-martingale problem \eqref{f06} with
$\mf L$, $\mc D_S$, replaced by $\mf L_B$, $\mc D_{B,B}$, respectively.
\end{proposition}

Fix $x\in \Sigma$ and recall the definition of the absorption times
$\sigma_n$ introduced in \eqref{sib}. By the strong Markov property
which, according to Proposition \ref{smark}, holds for all solutions
of the $\mf L$-martingale problem, on the set $\{\sigma_n < \infty\}$,
outside a $\bb P_x$-null set, any regular conditional probability
distribution of $\bb P_x$ given $\ms F_{\sigma_n}$ coincides with $\bb
P_{X_{\sigma_n}}$. Therefore, by Proposition \ref{sf01}, on the set
$\{\sigma_n < \infty\}$, after time $\sigma_n$, $X$ evolves on
$\Sigma_{\ms B_n}$ as the diffusion with generator $\mf L_{\ms B_n}$.

\smallskip\noindent{\bf \ref{rec02}.6. An alternative martingale
  problem.}  The previous informal description of the evolution of the
process after being absorbed at the boundary can be made rigorous by
the formulation of an alternative martingale problem, based on the
operators $\{\mf L_B\}$, for which $\{\bb P_x : x\in \Sigma\}$ are
also solutions. To define this martingale problem, we introduce an
operator $\ms L$ and a domain $D_0(\Sigma)$. For each $B\subseteq S$
with at least two elements, let
\begin{equation*}
\mathring{\Sigma}_B \;:=\; \{ x\in \Sigma_B : x_j>0\; \forall j\in B
\}\;, \quad \mathring{\Sigma}:=\mathring{\Sigma}_S\;.
\end{equation*}
In addition, given a function $F:\Sigma\to \bb R$ define
$[F]_B:\Sigma_B\to \bb R$ as
\begin{equation}
\label{cf03}
[F]_B(x) \;:=\; 
\left\{
\begin{array}{ll}
F(x,{\bf 0}) \;, & \textrm{if $x\in \mathring{\Sigma}_B$}\;;\\ 
0 \;, & \textrm{otherwise}\;.
\end{array}
\right.
\end{equation}
Note that $B$ is allowed to be equal to $S$, in which case $[F]_S(x) =
F(x) \, \mb 1\{x\in \mathring{\Sigma}\}$. In particular, $[F]_S$ may be
different from $F$ at the boundary $\{x\in\Sigma : x_j =0 \text{ for
  some $j$ } \}$.

Define $D_0(\Sigma)$ as the set of functions $F:\Sigma\to \bb R$ such
that, for all $B\subseteq S$ with at least two elements, $[F]_B$
belongs to $C^2(\Sigma_B)$ and has compact support contained in
$\mathring{\Sigma}_B$. Note that functions in the domain $D_0(\Sigma)$
are not continuous.

Recall that $x_{S_0}$, $S_0\subseteq S$, represents the coordinates of
$x$ in $S_0$. For all $F\in D_0(\Sigma)$, define $\ms LF:\Sigma \to
\bb R$ as $\ms LF(\bs e_j)=0$ for all $j\in S$ and
\begin{equation*}
\ms L F(x) \;:=\;(\mf L_B [F]_B) (x_B)\;,\quad \textrm{whenever $\ms
  B(x)=B$}\;, \quad x\not\in \{{\bs e}_j : j\in S\}\;.
\end{equation*}

To deal with the transitions between two consecutive time intervals in
$[\sigma_{n-1},\sigma_{n})$, $n\ge 1$, consider the jump process
\begin{equation*}
N_t \;:=\; \sup\{ n\ge 0 : \sigma_n\le t \} \;, \quad t\ge 0\;,
\end{equation*}
and define $N^{S}_t:=N_t\land |S|$, $t\ge 0$, so that $(N^S_t)_{t\ge
  0}$ is a bounded right-continuous non-decreasing $\ms F_t$-adapted
process. Clearly, if the measure $\bb P$ is absorbing,
\begin{equation*}
\bb P[ \textrm{$N_t=N^S_t$, for all $t\ge 0$} ] \;=\; 1\;.
\end{equation*}

Next theorem is proved in Section \ref{rec03}.

\begin{theorem}
\label{amp1}
For each $x\in \Sigma$ and any $F\in D_0(\Sigma)$,
\begin{equation*}
F(X_t) - \int_0^t  \ms LF(X_s)ds - \int_0^t F(X_s)dN^S_s\;, \quad t\ge 0
\end{equation*}
is a $\bb P_x$-martingale with respect to $(\ms F_t)_{t\ge 0}$.
\end{theorem}

The strong Markov property, assertion \eqref{fsig}, Theorem \ref{amp0}
and Theorem \ref{amp1} provide a precise picture of the dynamics
determined by $\{\bb P_x : x\in \Sigma\}$.

To fix ideas assume that $x\in \mathring{\Sigma}$. By assertion
\eqref{fsig}, the hitting time $\sigma_1$ is $\bb P_x$-a.s. finite:
\begin{equation*}
\bb P_x\big[ \sigma_1<\infty \big] \;=\; 1\;.   
\end{equation*}
Before hitting the boundary of $\Sigma$, the process $X_t$ evolves as
a diffusion process with bounded and smooth coefficients $\mb b$ and
$\mb a_{\rm s}$ (see Lemma \ref{unle} for a more precise
statement). This property characterizes the evolution on the time
interval $[0,\sigma_1]$. After $\sigma_1$, Theorem \ref{amp0} asserts
that coordinates in $\ms A_1$ remain equal to $0$ until a new
coordinate vanishes:
\begin{equation*}
\bb P_x \, \Big[ \sum_{j\in \ms A_1} X_{t}(j) = 0 
\,,\, \sigma_1 \le t < \sigma_2 \, \Big] \;=\; 1\;.
\end{equation*}
It follows from Theorem \ref{amp1} and from the strong Markov property
that, given $\{\ms B_1=B\}$, on the time interval $[\sigma_1 ,
\sigma_2)$ the coordinates of the process $X_t$ in $B$ evolve as the
original diffusion in which the Markov generator $\mf L$ is replaced
by $\mf L_B$. Furthermore, by the strong Markov property and
\eqref{fsig}, on $\{|\ms B|>1\}$, the hitting time $\sigma_2$ is $\bb
P_x$-a.s. finite. Iterating this argument, we obtain a complete
description of the path under the law $\bb P_x$: for each $n\ge 1$, on
$\{|\ms B_{n}|>1\}$, $\sigma_{n+1}$ is $\bb P_x$-a.s. finite, and on
each time interval $[\sigma_{n},\sigma_{n+1})$ the process $X_t$
evolves as a diffusion on lower and lower dimensional spaces
characterized by the generator $\mf L_{\ms B_{n}}$ where $(\ms
B_n)_{n\ge 0}$ turns to be a random decreasing sequence of subsets of
$S$. Eventually the process $X_t$ attains a point in $\{\bs e_j : j\in
S\}$. From this time on, according to observation \eqref{trap}, the
process remains trapped at this point for ever.

\smallskip\noindent{\bf \ref{rec02}.7. Remarks.}

\smallskip \noindent{\sl A. The case $|S|=2$:} When the set $S$ is a
pair, the diffusion $X_t$ can be mapped to a one-dimensional
diffusion. In this case, $\mc D$ corresponds to the set of twice
continuously differentiable functions $f:[0,1]\to \bb R$, such that
$f'(x)=f''(x)=0$ for $x=0,1$ and the respective generator $\mf L : \mc
D \to C([0,1])$ is given by
\begin{equation*}
(\mf L f)(x) \;=\; b \,\mb 1\{0<x<1\} \Big\{  \frac{M_2}{1-x}  \; -\;
\frac{M_1}{x} \Big\} f'(x) \; +\; 2(M_1+M_2) f''(x)\;.
\end{equation*}

\smallskip \noindent{\sl B. Wentzell boundary conditions:} The process
$X_t$ can be viewed as a diffusion with Wentzell boundary conditions
\cite[Section IV.7]{iw}. In order to do that we need to introduce a
differential operator for each boundary $\partial_A \Sigma := \{x\in
\Sigma : \sum_{j\in A} x_j =0 \,,\, \prod_{k\in A^c} x_k >0\}$ of
$\Sigma$.  It follows from the description of the process presented
above that the local time of the boundary $\partial_A \Sigma$ is equal
to $0$ until the process hits the boundary $\partial_A \Sigma$. From
this time until one coordinate in $A^c$ reaches $0$, the local time
strictly increases with slope $1$, and after this latter time the
local time remains constant for ever. Unfortunately, the results and
the techniques on diffusions with Wentzell boundary conditions do not
apply in our context because the drift explodes as the process
approaches the boundary.

\smallskip \noindent{\sl C. Empty sites:} Our proof does not preclude
the possibility that at time $\sigma_1$ more than one coordinate
vanishes. We believe that this event has $\bb P_x$-probability equal
to $0$, but we were not able to exclude it, and it does not play a
role in the argument.

\smallskip \noindent{\sl D. Terminology:} We refer to $\{\bb P_x :
x\in \Sigma\}$ as an ``absorbed'' diffusion to distinguish it from
``sticky'' diffusions \cite[Section IV.7]{iw}. While sticky diffusions
may reflect at the boundary, even if the local time at the boundary is
not identically equal to $0$, as observed above the process $X_t$
remains at the boundary once it hits it.

\smallskip \noindent{\sl E. Boundary conditions:} The empty
coordinates remain empty due to the strong drift. The diffusivity at
the boundary of $\Sigma$ does not vanish. In particular, the process
attempts to leave the boundary, but these attempts fail due to the
strong drift which keeps the diffusion at the boundary. Actually,
simulations show that there is a mesoscopic scale, between the
microscopic scale of the zero-range process and the macroscopic scale
of the absorbed diffusion, in which the process detaches itself from
the boundary.

\smallskip \noindent{\sl F. A model for concentration of wealth:} The
condensing zero-range processes introduced above have beeen used as a
model to describe jamming in traffic, coalescence in granular systems,
gelation in networks, and wealth concentration in macroeconomies
(\cite{eh} and references therein). 

\smallskip \noindent{\sl G. The parameter $b$:} It must be emphasized
that the parameter $b$ plays an important role. Condensation
(cf. \cite{gss, al, bl3} for the terminology) does not occur for
$b<1$. At $b=1$ condensation is expected to occur, but the time scale
in which the condensate evolves should have logarithmic corrections.
This means that for $b<1$ the diffusion whose generator is given by
\eqref{gens}, if it exists, is not expected to be absorbed at the
boundary.

\smallskip \noindent{\sl H. Asymptotic behavior as $L\to\infty$:} As
mentioned in the introduction, an interesting open problem consists in
describing the evolution of condensing zero-range processes as $N$ and
$L \to \infty$ when starting from a supercritical density profile. For
example, to prove the hydrodynamical behavior of the system if the
initial density profile $\rho_0: [0,1) \to \bb R_+$ is such that
$\rho_0(x)> \rho_c$ for all $x$, where $\rho_c$ is the critical
density (precisely defined in \cite{gss, al, l3}). An alternative open
problem, which might be more tractable, consists in proving the
scaling limit of the diffusion whose generator is given by
\eqref{gens}, in the case where $S = \bb T_L$ is the discrete
one-dimensional torus with $L$ points, and $r(j,k)$ the jump rates of
a symmetric, nearest-neighbor random walk on $\bb T_L$

\smallskip\noindent {\bf \ref{rec02}.8. The nucleation phase of
  condensing zero-range processes.}  Denote by $D(\bb R_+, \Sigma)$
the space of $\Sigma$-valued, right continuous trajectories with left
limits, endowed with the Skorohod topology, and by $\bb P^N_{x}$,
$x\in \Sigma_N$, the probability measure on $D(\bb R_+, \Sigma)$
induced by the Markov chain $X^{\scs N}_t$ starting from
$x$. Expectation with respect to $\bb P^N_x$ is represented by $\bb
E^N_x$.

\begin{theorem}
\label{mt2}
Let $x_N\in\Sigma_N$ be a sequence converging to $x\in \Sigma$. Then,
$\bb P^N_{x_N}$ converges to $\bb P_x$ in the Skorohod topology.
\end{theorem}

Note that for each $x\in \Sigma$, the measure $\bb P_x$ is
concentrated on the space $C(\bb R_+,\Sigma)$ of continuous
trajectories.

The proof of Theorem \ref{mt2} is divided in two steps.  We show in
Proposition \ref{proptight} that the sequence of probability measures
$\bb P^N_{x_N}$ is tight, and we prove in Proposition \ref{mp1} that
any limit point of the sequence $\bb P^N_{x_N}$ solves the martingale
problem \eqref{f06}. Of course, the existence part of Theorem
\ref{r03} follows from this result. It is worth remarking that in the
proof of tightness we do not need to require $b>1$.

\section{Harmonic Extension}
\label{sec01}

Fix a proper subset $B$ of $S$ with at least two elements and let
$A=B^c$. The main result of this section asserts that it is possible
to extend a smooth function $f:\Sigma_B \to \bb R$ to a function
$F:\Sigma \to \bb R$ in such a way that $F$ belongs to $\mc D_A$ and
$(\mf L F)(x) = (\mf L_B f)(x_B)$ for $x$ in the submanifold
\begin{equation*}
\Sigma_{B,0} \;:=\; \{x\in \Sigma : \sum_{i\in A} x_i =0\}\;.
\end{equation*}

\smallskip\noindent{\bf \ref{sec01}.A The Trace process.}
Let us start recalling the definition of the trace of a Markov chain
on a subset of its state space. We refer to \cite{bl2} for more
details.

Let $D(\bb R_+, S)$ be the set of right-continuous trajectories $e:
\bb R_+ \to S$ with left limits, endowed with the Skorohod
topology. Denote by $\mb P_j$, $j\in S$, the probability measure on
$D(\bb R_+, S)$ induced by the Markov chain $(x_t)_{t\ge 0}$ with jump
rates ${\bs r}=\{r(j,k) : j,k \in S \}$ and starting from $j$. Denote
by $T_{S_0}$ (resp. $T^+_{S_0}$), $S_0 \subseteq S$, the hitting time
of (resp. return time to) $S_0$:
\begin{equation*}
T_{S_0} \;=\; \inf\{ t\ge 0 : x_t\in S_0\}\;, \quad
T^+_{S_0} \;=\; \inf\{ t\ge \tau_1 : x_t\in S_0 \}\;,
\end{equation*}
where $\tau_1$ represents the time of the first jump: $\tau_1 =
\inf\{t\ge 0 : x_t \not = x_0 \}$. Fix a nonempty subset $B$ of $S$
with at least two elements and denote by $(x^B_t)_{t\ge 0}$ the trace
of $(x_t)_{t\ge 0}$ on $B$. This is the irreducible, $B$-valued Markov
chain whose jump rates, denoted by ${\bs r}^B=\{r^B(j,k) : j,k\in
B\}$, are given by
\begin{equation}
\label{11}
r^B(j,k) \;:=\; \lambda(j) \, \mb P_j [ T_{k} = T^+_{B}]\;, \quad k\; \not =
j\;\in\; B\;,
\end{equation}
and set $r^B(j,j)=0$, $j\in S$ for notational convenience.

\smallskip\noindent{\bf \ref{sec01}.B The functions ${\bs u}_k$.}
For each $k\in B$, let ${\bs u}_k = {\bs u}^B_k:S\to [0,1]$ be the
only $\mc L$-harmonic extension on $S$ of the indicator of $\{k\}$ on
$B$, i.e. ${\bs u}_k$ is the solution of
\begin{equation*}
\left\{
\begin{array}{ll}
{\bs u}_k(j) = \delta_{j,k}, & \textrm{for $j\in B$}; \\
\mc L \bs u_k(j)=0, & \textrm{for $j\in S\setminus B$}.
\end{array}
\right.
\end{equation*}
Actually, the vectors $\{\bs u_k : k\in B\}$ can also be written as
probabilities:
\begin{equation}
\label{15}
\mb P_j [ T_{k} = T_{B}] \;=\; \bs u_k(j) \;, \quad \forall \, k\in B \;,\;
j\in S\; .
\end{equation}
In particular, by using the strong Markov property in \eqref{11} we
get the relation
\begin{equation*}
r^B(j,k) \;=\; r(j,k) \;+\; \sum_{\ell\in B^c} r(j, \ell) \, \mb P_\ell [
T_k = T_{B}]\; =\; \sum_{\ell\in S} r(j,\ell) \, \bs u_k(\ell)\;,
\;\;\text{for $k\not = j \in B$}\;.
\end{equation*}

\smallskip\noindent{\bf \ref{sec01}.C Relation between ${\bs u}_k$ and
  ${\bs v}^B_j$.}  Recall the definition of the vectors $\{\bs v_j^B :
j\in B\}$ introduced in \eqref{vbj}.  We claim that
\begin{equation}
\label{10}
{\bs v}^B_j(k) \;=\; \mc L {\bs u}_k(j) \quad \textrm{for all $j,k \in B$}\;.
\end{equation}
Indeed, on the one hand, by definition of $\bs v_j^B$, and by the last
identity of the previous subsection,
\begin{equation*}
{\bs v}^B_j(k) \;=\; r^B(j,k) \;=\; \mc L {\bs u}_k(j)\;,
\quad \textrm{for $k\not = j \in B$}\;.
\end{equation*}
On the other hand, for any $k\in B$,
\begin{equation*}
\sum_{j\in B}m_j{\bs v}^B_j(k) \;=\; 0 \;=\; 
\sum_{j\in B} m_j\mc L {\bs u}_k(j) \;.
\end{equation*}
The first identity follows from the fact that $\bs m$ restricted to
$B$ is also an invariant measure for $\bs r^B$. For the second
equality, as $\mc L {\bs u}_k(j) =0$ for $j\not\in B$, observe that
$\sum_{j\in B} m_j\mc L {\bs u}_k(j) = \sum_{j\in S} m_j\mc L {\bs
  u}_k(j) = 0$ because $\bs m$ is an invariant measure for $\bs
r$. The two previous displayed equations yield claim \eqref{10}.

Let $\mc L^B$ stand for the generator corresponding to the jump rates
$\bs r^B$. Then, for any $j,k\in B$, ${\bs v}^B_j(k)$ equals $\mc
L^B\bs e_k(j)$, where to keep notation simple, we let $\{\bs e_k: k\in
B\}$ stand for the canonical basis of $\bb R^B$. Thus, \eqref{10} can
also be written as
\begin{equation}
\label{uek}
\mc L^B \bs e_k  \;\equiv\; \mc L {\bs u}_k \quad \textrm{on $B$
for any $k\in B$}\;.
\end{equation}

\smallskip\noindent{\bf \ref{sec01}.D The projection $\Upsilon^B$.} 
Recall the definition of the submanifold $\Sigma_{B,0}$ introduced at
the beginning of this section. Denote by $\Upsilon=\Upsilon^B:\Sigma
\to \Sigma_B$ the linear map given by
\begin{equation}
\label{bsu}
[\Upsilon (x)]_k \;=\; \bs u_k \cdot x \;=\; x_k + 
\sum_{j\in A} \bs u_k(j) x_j \;, \quad k\in B \;.
\end{equation}
It is easy to check that $\Upsilon(x)\in \Sigma_B$ for all $x\in
\Sigma$, and that
\begin{equation*}
\Upsilon(x) \;=\; x_B \; \text{ for all } x\in \Sigma_{B,0} \;,
\end{equation*}
where $x_B$ stands for the coordinates of $x$ in $B$. 

We claim that
\begin{equation}
\label{uvuv}
\Upsilon(\bs v_j) \;=\; \bs v^B_j \;, \quad j\in B\;.
\end{equation}
Indeed, by definition of $\Upsilon$ and of the vectors $\{\bs v_j : j\in
S\}$,
\begin{equation}
\label{ff02}
\Upsilon({\bs v}_j)(k) \;=\; {\bs u}_k \cdot {\bs v}_j \;=\; 
\mc L {{\bs u}_k}(j)\;,\quad j \in S \,,\; k\in B \;.
\end{equation}
By \eqref{10}, the last expression equals $\bs v_j^B(k)$ proving the
assertion.

\smallskip\noindent{\bf \ref{sec01}.E Harmonic extensions.}  For a
nonempty subset $B$ of $S$, denote by $\bs m_B$ the restriction of the
measure $\bs m$ on $B$: $m_B(j) = m(j)$, $j\in B$. Denote by $\mc L^*$
(resp. $\mc L^{B,*}$) the adjoint of the generator $\mc L$ (resp. $\mc
L^{B}$) in $L^2(\bs m)$ (resp. $L^2(\bs m_B)$), and by $\mc S$
(resp. $\mc S^{B}$) the symmetric part of $\mc L$ (resp. $\mc L^{B}$):
$\mc S = (1/2)(\mc L + \mc L^*)$ (resp. $\mc S^B = (1/2)(\mc L^B + \mc
L^{B,*})$).

Given a function $f\in C^2(\Sigma_B)$, define $F\in C^2(\Sigma)$ as
$F=f\circ \Upsilon$.

\begin{lemma}
\label{s04}
The function $F$ is an extension of $f$ in the sense that
\begin{equation}
\label{ext1}
F(x) \;=\; f(x_B) \;, \quad \forall x\in \Sigma_{B,0}\;.
\end{equation}
Moreover,
\begin{equation}
\label{jina}
{\bs v}_j \cdot \nabla F(x) = 0 \;,\quad j\in A \,,\; x\in \Sigma\;,
\end{equation}
so that $F\in \mc D_A$. Finally, 
\begin{equation} 
\label{eq04}
\mf L F (x)\;=\;\mf L_B f(x_B)\;,\quad  x\in\mathring{\Sigma}_{B,0}\;,
\end{equation}
where
\begin{equation*}
\mathring{\Sigma}_{B,0} \;:=\; \{ x\in \Sigma_{B,0} : x_j>0 \; 
j\in B \}\;. 
\end{equation*}
\end{lemma}

\begin{proof}
The first assertion of the lemma follows from the displayed equation
below \eqref{bsu}. We turn to the second assertion. Fix an arbitrary
$x\in \Sigma$. By definition of $\Upsilon$, $\partial_{x_k}
\Upsilon(x)(\ell) = \bs u_\ell (k)$ for all $\ell\in B$, $k\in
S$. Hence, by definition of $F$ and by \eqref{ff02}, 
\begin{equation}
\label{ob1}
{\bs v}_j\cdot \nabla F(x) \;=\;  \Upsilon({\bs v}_j)\cdot 
[(\nabla f) (\Upsilon(x))]\;, \quad j\in S\;. 
\end{equation}
If $j$ belongs to $A$, by \eqref{ff02} and by definition of ${\bs
  u}_k$, $\Upsilon({\bs v}_j) (k) = \mc L {{\bs u}_k}(j) =0$ for all
$k\in B$.  This completes the proof of the second assertion.

We turn to the proof of the last assertion of the lemma. Fix $x\in
\mathring{\Sigma}_{B,0}$ and so $\Upsilon(x)=x_B$. We first examine
the first-order terms. By \eqref{ob1} and \eqref{uvuv} we have
\begin{equation*}
{\bs v}_j\cdot \nabla F(x) \;=\; \bs v^B_j \cdot \nabla f(x_B) \;, 
\quad j\in B\;.
\end{equation*}
Therefore, by \eqref{jina} and this last identity we conclude that the
first-order part of $(\mf L F)(x)$ equals the first-order part of
$(\mf L_B f)(x_B)$. 

It remains to examine the second-order terms of the generators. The
second-order piece of $(\mf L F)(x)$ is
\begin{equation*}
\frac 12 \sum_{j,k\in S} m_j r(j,k)(\partial_{x_k} - \partial_{x_j})^2
(f\circ \Upsilon)(x) \;.
\end{equation*}
A simple computation shows that
\begin{equation*}
(\partial_{x_k} - \partial_{x_j})^2(f\circ \Upsilon)(x) \;=\; 
\sum_{i,\ell\in B} \partial^2_{x_{\ell},x_{i}}f(x_B)
\{ {\bs u}_{\ell}(k) - \bs u_{\ell}(j) \} 
\{ {\bs u}_{i}(k) - \bs u_{i}(j) \}\;.
\end{equation*}
In view of this identity, interchanging the sums, the penultimate
displayed equation becomes
\begin{equation}
\label{sot2}
\sum_{i,\ell\in B} (\partial^2_{x_{\ell},x_{i}}f) (x_B)\,
\langle \bs u_{\ell} , -\mc S \bs u_{i} \rangle_{\bs m} \;, 
\end{equation}
where, recall, $\mc S$ represents the symmetric part of the generator
$\mc L$ in $L^2(\bs m)$.

Fix $i$, $\ell\in B$. Since $\mc L \bs u_{i}$ vanishes on $A=B^c$, 
\begin{equation*}
\langle \bs u_{\ell} , -\mc L \bs u_{i} \rangle_{\bs m}  \;=\;
\sum_{j\in B} \bs u_{\ell} (j) \, (-\mc L \bs u_{i}) (j) \, \bs m (j)
\end{equation*}
On the set $B$, $\bs u_{\ell}$ and $\bs e_{\ell}$ coincide, while, by
\eqref{uek}, $\mc L \bs u_{i} = \mc L^B \bs e_{i}$. Hence, the
penultimate formula is equal to
\begin{equation*}
\sum_{i,\ell\in B} (\partial^2_{x_{\ell},x_{i}}f) (x_B)\,
\langle \bs e_{\ell} , -\mc S^B \bs e_{i} \rangle_{\bs m_B} \;, 
\end{equation*}
where the scalar product is now performed over $B$. This expression is
equal to
\begin{equation*}
\frac{1}{2} \sum_{i,\ell\in B} \partial^2_{x_{\ell},x_{i}}f(x_B) 
\sum_{j,k\in B} m_jr^B(j,k)\{ \bs e_{\ell}(k) - \bs e_{\ell}(j) \} 
\{ \bs e_{i}(k) - \bs e_{i}(j) \}\;.
\end{equation*}
By interchanging the sums, this term becomes
\begin{equation*}
\frac{1}{2} \sum_{j,k\in B} m_j \, r^B(j,k) \,
(\partial_{x_{k}}- \partial_{x_{j}})^2 f(x_B) \;.
\end{equation*}
This is exactly the second-order term of $(\mf L_B f)(x_B)$. Last
assertion of Lemma \ref{s04} is hence proved.
\end{proof}

\section{The domain of the generator}

The proof that all solutions of the $\mf L$-martingale problem are
absorbed at the boundary, relies on the existence of super-harmonic,
non-negative functions which are strictly positive at the boundary.
The goal of this section is to provide such functions.

This is achieved by introducing in \eqref{fg1} a class of non-negative
functions and by applying Lemmata \ref{g1} and \ref{r01}.  Lemma
\ref{g1} states that it is possible to extend certain functions
$f:\Sigma_B \to \bb R$ which belong to $\mc D_D$, $D\subset B$, to
functions $F:\Sigma \to \bb R$ which belong to $\mc D_{D\cup B^c}$,
while Lemma \ref{r01} states that it is possible to modify a function
$F:\Sigma \to \bb R$ which belongs to $\mc D_{A}$ in a neighborhood of
the set $\{x\in \Sigma : \prod_{j\in A^c} x_j=0\}$ to transform in into
a function which belongs to $\mc D_S$.

We start in Lemma \ref{propi} below by defining a class of functions
$I_A$, $\varnothing \subsetneq A \subsetneq S$, which belong to $\mc
D_A$. These functions play a key role in the argument and they have to
be interpreted as smooth perturbations of the maps $x\mapsto
\sum_{i\in A} x^2_i$.

For a nonempty subset $S_0$ of $S$, let
\begin{equation*}
\Vert x\Vert_{S_0} \;:=\; \Big( \sum_{j\in S_0}x_j^2\Big)^ {1/2}\;, 
\quad x\in \Sigma \;,
\end{equation*}
and set $\Vert \cdot \Vert := \Vert \cdot \Vert_{S}$.

\begin{lemma}
\label{propi}
Let $A$ be a nonempty, proper subset of $S$.  There exists a
nonnegative, smooth function $I_A: \Sigma \to \bb R$ in $\mc D_A$, and
constants $0<c_1 < C_1 < \infty$, such that for all $x\in \Sigma$,
\begin{equation}
\label{35}
c_1 \Vert x\Vert_A \;\le\; \sqrt{I_A(x)} \;\le\;
C_1 \Vert x\Vert_A\;.
\end{equation}
Furthermore, for $x,y\in \Sigma$,
\begin{equation}
\label{sci13}
I_A(y) \;=\; \alpha^2 I_A(x) \; \textrm{ if $y_A=\alpha x_A$ for some $\alpha\ge 0$}\;.
\end{equation}
In particular, $I_A(x)$ only depends on $x_A$ for all $x\in \Sigma$.
\end{lemma}

The proof of this lemma is postponed to the last section of this
article. The function $I_A(x)$ has to be understood as a perturbation
of the function $x\mapsto \Vert x \Vert^2_A$ to turn this latter
function an element of $\mc D_A$. Let
\begin{equation}
\label{ja1}
J_A(x)\;=\; \sqrt{I_A(x)}\;, \quad x\in \Sigma \;,
\end{equation}
Of course, $J_A$ is smooth on $\{x\in \Sigma : \Vert
x\Vert_A>0\}$. Furthermore, by \eqref{sci13}, for $x,y\in \Sigma$ such
that $x_A \not = 0$, $y_A=\alpha x_A$ for some $\alpha > 0$,
\begin{equation}
\label{ja2}
\nabla J_A(y) \;=\; \nabla J_A(x) \quad {\rm and} \quad 
\Vert y\Vert_A\, {\rm Hess}\, J_A(y) \;=\; \Vert x\Vert_A \, 
{\rm Hess} \,J_A(x)\;.
\end{equation}
In particular, 
\begin{equation}
\label{jd}
\sup_{\Vert x\Vert_A>0} \Vert \nabla J_A(x) \Vert \;<\; \infty
\; \text{ and }\;
\sup_{\Vert x\Vert_A>0} \Vert x\Vert_A \, |\partial^2_{x_j,x_k}
J_A(x)| \;<\; \infty \;,\quad j,k\in S \;.
\end{equation}
We shall use the following estimate in Lemma \ref{g1} below.

\begin{lemma}
\label{g0}
For all $k\in A$,
\begin{equation*}
\sup_{\Vert x\Vert_A>0}  \frac{|\bs v_k \cdot \nabla J_A(x)|}{x_k} 
\, \Vert x\Vert_A \;<\; \infty \;.
\end{equation*}
\end{lemma}

\begin{proof}
Fix $k\in A$. For every $x\in \Sigma$ such that $\Vert x\Vert_A>0$, 
\begin{equation*}
\sum_{j\in A}x(j) \;<\; 2 |A|^{1/2}\Vert x\Vert_A\;.
\end{equation*}
In particular, for each $x\in \Sigma$ such that $\Vert x\Vert_A>0$
there exists some $z\in \Sigma$ such that 
\begin{equation*}
z_A \;=\; \frac{x_A}{2 |A|^{1/2}\Vert x\Vert_A}\;\cdot
\end{equation*}
In view of \eqref{ja2}, for this choice we have that
\begin{equation*}
\frac{|\bs v_k \cdot \nabla J_A(x)|}{x_k} \, \Vert x\Vert_A 
\;=\;  \frac{1}{2|A|^{1/2}} \, \frac{|\bs v_k \cdot \nabla J_A(z)|}{z_k} \;.
\end{equation*}
Therefore,
\begin{eqnarray*}
\sup_{\Vert x\Vert_A>0}  \frac{|\bs v_k \cdot \nabla J_A(x)|}{x_k} \, 
\Vert x\Vert_A &\le& \frac{1}{2|A|^{1/2}} 
\sup_{\Vert z\Vert_A= (4|A|)^{-1/2}} \frac{|\bs v_k \cdot \nabla J_A(z)|}{z_k} \\
&=& \frac{1}{2|A|^{1/2}} \sup_{\Vert z\Vert_A= (4|A|)^{-1/2}} 
\frac{|\bs v_k \cdot \nabla I_A(z)|}{2 J_A(z)\, z_k} \\
&\le& \frac{1}{2c_1} \sup_{z_k\not = 0} 
\frac{|\bs v_k \cdot \nabla I_A(z)|}{z_k}\;, 
\end{eqnarray*}
where we used estimate \eqref{35} in the last inequality. The last
expression is finite because $I_A\in \mc D_A$, which completes the
proof.
\end{proof}

Fix a nonempty subset $B$ of $S$ and let $A=B^c$. Suppose now that in
Lemma \ref{s04}, $f:\Sigma_B\to \bb R$ is of the special form
\begin{equation}\label{fg1}
f(x) \;=\; \prod_{j\in D} x_j^{p+1} \;,\quad x\in \Sigma_B\;,
\end{equation}
for some nonempty subset $D$ of $B$ and for some $p>1$. In this case,
we may improve Lemma \ref{s04} obtaining and extension $F$ of $f$
which belongs to $\mc D_{A\cup D}$. From now on, for each nonempty
subset $S_0$ of $S$, let
\begin{equation}
\label{43}
\pi_{S_0} (x) \;=\;  \prod_{j\in S_0} x_j \;.
\end{equation}

\begin{lemma}
\label{g1}
Let $f:\Sigma_B \to \bb R$ be given by \eqref{fg1} for a nonempty
subset $D$ of $B$ and $p>1$. Then there exists a function $F:\Sigma
\to \bb R$ in $\mc D_{A\cup D}$ satisfying \eqref{ext1} and
\eqref{eq04}. Furthermore, if $\pi_D(x)=0$ then $\mf LF(x)=0$.
\end{lemma}

\begin{proof}
If $A=\varnothing$, it is easy to check that $F=f$ satisfies all the
requirements. We then assume $A$ is nonempty. Let $\Psi:\bb R \to \bb
R$ be a non-increasing function in $C^2(\bb R)$ which is equal to $1$
on $(-\infty,0]$, and is equal to $0$ on $[1,\infty)$. For example,
the function which on the interval $[0,1]$ is given by
\begin{equation}
\label{psi0}
\Psi(a) \;=\; (1-a)^3 (1+3a + 6a^2)\;, \quad a\in [0,1]\;.
\end{equation}
We fix the constant 
\begin{equation*}
\beta \;:=\; \frac{2\big( 1+|A|^{1/2} \big) }{c_1} \;,
\end{equation*}
where $c_1>0$ is the constant given in \eqref{35}. 

Let $F: \Sigma \to \bb R$ be defined by
\begin{equation*}
F(x) \;=\; \left\{
\begin{array}{ll}
\displaystyle (f \circ \Upsilon) (x)\, 
\Psi \Big( \frac{\beta\, J_A(x)}{\pi_D( \Upsilon(x))}
- 1 \Big) \;, & \textrm{if $\pi_D(
\Upsilon(x))>0$}\;;\bigskip \\
0 \;, & \textrm{otherwise}\;, 
\end{array}
\right.
\end{equation*}
where $\Upsilon: \Sigma\to \Sigma_B$ is the linear map introduced in
\eqref{bsu}. Note that
\begin{equation*}
F(x) \;=\; (f\circ \Upsilon)(x) \quad\text{if $\pi_D(\Upsilon(x))=0$}
\end{equation*}
because both expressions vanish when $\pi_D(\Upsilon(x))=0$.

From the definition of $\Psi$ it easily follows that
\begin{equation*}
F(x) \;=\; f(\Upsilon(x)) \;=\; f(x_B)\;, \quad
\textrm{for all $x\in \Sigma_{B,0}$}\;,
\end{equation*}
proving that $F$ satisfies \eqref{ext1}. 

\smallskip\noindent{\bf A. $F$ belongs to $C^1(\Sigma)$.} Denote by $\ms
V$ the open subset $\{ x\in \Sigma : \pi_D(\Upsilon(x))>0 \}$ and let
\begin{equation*}
R(x) \;=\; \beta\, J_A(x)/\pi_D(\Upsilon(x)) \quad 
\textrm{for $x\in \ms V$}\;.
\end{equation*}
It is simple to check that $\Psi(R-1)$ and $F$ are of class $C^2$ in
$\ms V$. In particular, to prove that $F$ belongs to $C^2(\Sigma)$, we
only need to examine the behavior of the derivatives of $F$ close to
the boundary of $\ms V$.

We claim that 
\begin{equation}
\label{cl0}
\Vert \nabla F(x)\Vert \;\le \; C\, \Big\{ \Vert \nabla f(w) \Vert_B
\;+\;  \pi_D(w)^{p} \Big\}\;,\quad  x\in \ms V\;,
\end{equation}
where $w$ stands for $\Upsilon(x)$, and $C$ for a finite positive
constant, independent of $x$, and which may be different from line to
line.

On the one hand, since $\Psi\equiv 1$ on $(-\infty,0]$, by \eqref{35},
it is clear that
\begin{equation*}
\textrm{$\Psi(R-1)\equiv 1$ on the open subset 
$\big\{x\in \Sigma : \beta C_1\Vert x\Vert_A < \pi_D( \Upsilon(x))\big\} $ }\;,
\end{equation*}
so that \eqref{cl0} holds in this open subset of $\ms V$.  On the
other hand, if $x$ is a point in $\ms V$ such that $\Vert x\Vert_A
>0$, by \eqref{jd} and by the fact that $\Psi'(R(x)-1)=0$ for $R(x)\ge
2$,
\begin{equation*}
\Vert \nabla \{ \Psi(R-1) \}(x) \Vert \;\le \; \frac{C}{\pi_D(\Upsilon(x))}  
\;,\quad \forall x\in \ms V \;,
\end{equation*}
which proves that \eqref{cl0} also holds in this case, in view of the
definition of $f$.

By \eqref{cl0} and by the definition of $f$, $\nabla F(x)$ vanishes as
$x\in\ms V$ approaches the boundary of $\ms V$.  On the other hand,
since $\Psi$ is bounded, it follows from the definition of $f$ that
\begin{equation}
\label{cl1}
\partial_{x_j} F(x)=0\;, \quad \forall j\in S\,,\; x\in
\Sigma\setminus \ms V \;, 
\end{equation}
which proves that $F$ belongs to $C^1(\Sigma)$.

\smallskip\noindent{\bf B. $F$ belongs to $C^2(\Sigma)$.} We claim
that there exists a finite constant $C$ such that for all $j,k\in S$,
and all $x\in \ms V$,
\begin{equation}
\label{cf01}
\big| (\partial^2_{x_jx_k}) F(x)\big| \;\le\; C\,
\pi_D(\Upsilon(x))^{p-1}\;.
\end{equation}
Indeed, for $x$ in the open set $\{x\in \Sigma : \beta C_1\Vert
x\Vert_A < \pi_D( \Upsilon(x))\}$, $F(x)= f(\Upsilon(x))$ and the
assertion is easily proved. Additionally, for $x\in \ms V$ such that
$\Vert x\Vert_A>0$, by \eqref{35}, by \eqref{jd}, and by the fact that
$\Psi'(R(x)-1)=\Psi''(R(x)-1)=0$ for $R(x)\ge 2$ and for $R(x)\le 1$,
\begin{equation*}
\partial^2_{x_j x_k}\{\Psi(R-1)\}(x) \;\le \; 
\frac{C}{\pi_D(\Upsilon(x))^2}\;\cdot
\end{equation*}
Assertion \eqref{cf01} for $x\in \ms V$ such that $\Vert x\Vert_A>0$
is a simple consequence of this estimate, of the bound on the first
derivative of $\Psi(R-1)$ obtained in part {\bf A} of the proof, and
of the definition of $f$.

We claim that
\begin{equation}
\label{cl2}
\partial^2_{x_jx_k}F(x) \;=\;0\;, \quad \forall j,k\in S\,,\;  
x\in \Sigma\setminus \ms V\;.
\end{equation}
Indeed, fix $x_0\in \Sigma$ such that $\pi_D(\Upsilon(x_0))=0$,
so that $F(x_0)=\partial_{x_j}F(x_0)=0$ for all $j\in S$. By
\eqref{cl0} and \eqref{cl1} we have
\begin{equation*}
\frac{\Vert \nabla F(x)\Vert}{\Vert x - x_0\Vert} \;\le \; 
C \Big\{ \frac{\Vert (\nabla f) (\Upsilon(x))\Vert_B}{\Vert x -x_0 \Vert} 
\;+\; \frac{ \pi_D(\Upsilon(x))^p }{\Vert x -x_0 \Vert } \Big\}\;, 
\quad \textrm{for $x\not = x_0$}\;.
\end{equation*}
Since $\nabla f(w)$, $\pi_D(w)^p$, ${\rm Hess} \, f(w)$, and $\nabla
\pi_D(w)^p$ vanish at $w=\Upsilon(x_0)$,
\begin{equation*}
\Vert (\nabla f)(\Upsilon(x))\Vert_B \;+\; \pi_D(\Upsilon(x))^p 
\;\le\; C \,\Vert \Upsilon(x) - \Upsilon(x_0) \Vert_B^2 
\;\le\; C \,\Vert x-x_0\Vert^2 \;,
\end{equation*}
for all $x\in \Sigma$, which proves \eqref{cl2}, and, in view of
\eqref{cf01}, that $F$ belongs to $C^2(\Sigma)$.

\smallskip\noindent{\bf C. $F$ satisfies \eqref{eq04}.} In order to prove
this property, observe that the functions $F$ and $f\circ \Upsilon$
coincide on
\begin{equation*}
\{ x\in \Sigma : \beta J_A(x) < \pi_D(\Upsilon(x)) \}\;.
\end{equation*}
Since $J_A$ and $\pi_D\circ \Upsilon$ are continuous functions, this
is an open subset of $\Sigma$. Moreover, $\mathring{\Sigma}_{B,0}$ is
contained in this open subset. Therefore, for any $x\in
\mathring{\Sigma}_{B,0}$, $\mf L F (x) = \mf L (f\circ \Upsilon)(x)$.
By Lemma \ref{s04}, this latter term is equal to $\mf L_Bf(x_B)$.

\smallskip\noindent{\bf D. $F$ belongs to $\mc D_D$.}  By \eqref{cl1},
$\bs v_\ell\cdot \nabla F(x) =0$ for $x\in \Sigma\setminus \ms V$ and
$\ell\in S$. Thus, in the proof that $F$ belongs to $\mc D_{D\cup A}$,
in the limit appearing in \eqref{cf02}, we only need to consider
points $x$ in $\ms V$. This is assumed below and in the Step {\bf E}
without further comment.

Fix $j\in D$. By \eqref{35} and by Cauchy-Schwarz inequality in $\bb
R^A$,
\begin{equation}
\label{37}
R(x)\;\ge\;\frac{\beta\, J_A(x)}{\Upsilon(x)(j)} \;\ge\;
\frac{\beta\, c_1 \Vert x\Vert_A} 
{x_j + \sum_{k\in A} {\bs u}_j(k) x_k}\;\ge\;
\frac{\beta\, c_1 \Vert x\Vert_A}{x_j + |A|^{1/2}\Vert x \Vert_A}\;\cdot
\end{equation}
Hence, by definition of $\beta$, $R(x) > \beta c_1/[1+|A|^{1/2}] = 2$
for $x$ such that $x_j<\Vert x\Vert_A$. In particular, by definition
of $\Psi$,
\begin{equation*}
\textrm{$F\equiv 0$\, on the open subset\, 
$\{x\in \Sigma : x_j<\Vert x\Vert_A\}$} \;. 
\end{equation*}
It follows from this observation and from \eqref{cl0} that
\begin{equation*}
|\bs v_j \cdot \nabla F(x)| \;\le\; C \big( \Vert \nabla f(w) \Vert_B 
\;+\; \pi_D(w)^p \big) \, {\mb 1}\{\Vert x\Vert_A \le x_j \} \;,
\quad x\in \ms V \;,
\end{equation*}
where $w=\Upsilon(x)$. For $x\in\Sigma$ such that $\Vert
x\Vert_A \le x_j$, 
\begin{equation*}
w_j \;:=\; \Upsilon(x)(j) \;\le\; x_j + |A|^{1/2}
\Vert x\Vert_A \;\le\; (1+|A|^{1/2}) x_j \;.
\end{equation*}
Therefore, by the next to the last displayed formula and by definition of $f$, 
\begin{equation*}
\begin{split}
\frac{|\bs v_j \cdot \nabla F(x)|}{x_j} \; & \le\; C \,
\Big\{ \frac{\Vert \nabla f(w) \Vert_B}{w_j} \;+\; 
\frac{\pi_D(w)^p}{w_j } \Big\} \, {\mb 1}\{\Vert x\Vert_A \le x_j \} 
\\
&\le\;  C\, \pi_D(w)^{p-1} \,{\mb 1}\{\Vert x\Vert_A \le x_j \}\;, 
\quad x\in \ms V \,,\; x_j\not=0 \;.
\end{split}
\end{equation*}

Fix $y\in \Sigma$ such that $y_j=0$. If $\Vert y\Vert_A>0$, in view of
the indicator in the previous estimate and by the remark formulated at
the beginning of this step,
\begin{equation*}
\lim_{\substack{x\to y \\  x_j>0}} 
\frac{|\bs v_j \cdot \nabla F(x)|}{x_j} \;=\; 0 \;.
\end{equation*}
In contrast, if $\Vert y\Vert_A=0$, $\pi_D(\Upsilon(y))=0$ because
$\Upsilon(y)(j)=0$. Hence, the same conclusion holds because
$\pi_D(w)$ converges to $\pi_D(\Upsilon(y))$ and $p>1$.  This
concludes the proof that $F$ belongs to $\mc D_D$.

\smallskip\noindent{\bf E. $F$ belongs to $\mc D_A$.} Recall from the
previous step that we may restrict our analysis to points $x$ in
$\ms V$. Fix $k\in A$ and $y\in \Sigma$ such that $y_k=0$.

Identity \eqref{jina} for the functions $\pi_D\circ \Upsilon$ and
$f\circ \Upsilon$ yield that
\begin{equation*}
\bs v_k\cdot \nabla F(x) \;=\; \beta \, \pi_D(\Upsilon(x))^{p} \,  
\Psi'(R(x)-1)\, \bs v_k\cdot \nabla J_A (x) \;,
\end{equation*}
for all $x\in \ms V$ such that $\Vert x\Vert_A>0$. 

We consider separately three cases which all rely on the identity
appearing in the previous displayed formula. Assume first that $\Vert
y\Vert_A>0$.  In this case, since $\nabla J_A (x) = \nabla I_A
(x)/2J_A(x)$, and since $I_A$ belongs to $\mc D_A$, in view of
\eqref{35},
\begin{equation*}
\lim_{\substack{x\to y\\ x_k>0}} 
\frac{\bs v_k\cdot \nabla F(x)}{x_k} \;=\; 0 \;. 
\end{equation*}

Next, assume that $\Vert y\Vert_A=0$ and that
$\pi_D(\Upsilon(y))>0$. In this case, $\Psi'(R(x)-1)=0$ in a
neighborhood of $y$, so that $\bs v_k\cdot \nabla F$ vanishes in a
neighborhood of $y$ in $\ms V$.

It remains to consider the case in which $\Vert y\Vert_A=0$ and
$\pi_D(\Upsilon(y))=0$. For $x\in \ms V$ such that $x_k>0$, by the
identity appearing in the first displayed equation of this step, by
Lemma \ref{g0}, and by definition of $\Psi$,
\begin{equation*}
\Big|\, \frac{\bs v_k\cdot \nabla F(x)}{x_k}\,\Big| \;\le\;
C \, \frac{\pi_D(\Upsilon(x))^p}{\Vert x\Vert_A} {\mb 1}\{R(x)\ge 1\} 
\;\le\; C\,  \pi_D(\Upsilon(x))^{p-1}\;,
\end{equation*}
where we used estimate \eqref{35} and the definition of $R$ in the
last inequality. As $x\to y$, the right hand side converges to
$\pi_D(\Upsilon(y))^{p-1}=0$ because $p>1$. This completes the proof
that $F$ belongs to $\mc D_A$.

\smallskip\noindent{\bf F. $\mf LF(x)=0$ if $\pi_D(x)=0$.}  Fix $x\in
\Sigma$ such that $\pi_D(x)=0$. If $\Vert x\Vert_A=0$ then
$\pi_D(\Upsilon(x))=\pi_D(x)=0$ and so $\mf LF(x)=0$ in view of
\eqref{cl1} and \eqref{cl2}. If $\Vert x\Vert_A>0$, $F$ vanishes in a
neighborhood of $x$ because so thus $\Psi(R(x)-1)$. In particular,
$\mf LF(x)=0$.
\end{proof}

We conclude this section by proving in Lemma \ref{r01} below that a
function $F$ in $\mc D_A$, $\varnothing \subsetneq A \subsetneq S$,
can be slightly modified into a function $H$ in $\mc D_S$.  For each
nonempty subset $S_0\subseteq S$ and $\epsilon>0$, let
\begin{equation}
\label{laep}
\Lambda_{S_0}(\epsilon) \;:=\; \{x \in \Sigma : 
\min_{j\in S_0}x_j \ge \epsilon\} \;.
\end{equation}

\begin{lemma}
\label{r01}
Fix a nonempty, proper subset $A$ of $S$ and a function $F$ in $\mc
D_A$. For every $\epsilon>0$ there exists a function $H$ in $\mc D_S$
such that
\begin{equation*}
F(x)\;=\;H(x) \quad {\rm and} \quad \mf LF(x)\;=\;\mf LH(x)\;, 
\quad  x\in \Lambda_{B}(\epsilon)\;,
\end{equation*}
where $B=A^c$. Moreover, if $F \in \mc D_A$ is such that $[F]_B$ has a
compact support contained in $\mathring{\Sigma}_B$ then, there exists
some $\epsilon>0$ and $H\in \mc D_S$ satisfying the previous identities
and such that
\begin{equation*}
H(x) \;=\; F(x) \;, \quad x\in \Sigma_{B,0}\;.
\end{equation*}
\end{lemma}

\begin{proof}
Recall function $\Psi$ defined in \eqref{psi0}. Let $\Phi(r) = 1 -
\Psi(r/3)$, $r\in \bb R$, so that $\Phi$ is a non-decreasing $C^2$
functions such that $\Phi(r) =0$ for $r\le 0$, and $\Phi(r)=1$ for
$r\ge 3$.  Let $\alpha = c_1/4C_1$, where $c_1$, $C_1$ have been
introduced in \eqref{35} and fix some arbitrary $\epsilon>0$. For
$k\in B$, let $G_k : \Sigma \to \bb R$, be given by
\begin{equation*}
G_k(x) \;=\; \phi_{k,\varnothing}(x) \prod_{\substack{D\subseteq A, \\ |D|=1}} \phi_{k,D}(x)
\prod_{\substack{D\subseteq A, \\ |D|=2}} \phi_{k,D}(x) \; \cdots \; \phi_{k,A}(x)\;.
\end{equation*}
In this formula,
\begin{equation*}
\phi_{k,D}(x) \;=\; \Phi \Big( \frac{ 9 \alpha^{2|D|} I_{D\cup\{k\}} (x)}
{\epsilon^2} -1 \Big) \;, \quad \textrm{for each $\varnothing \subseteq D \subseteq A$}\;.
\end{equation*}

The proof of the lemma relies on the elementary properties of the
functions $G_k$ and $\phi_{k,D}$ listed below.  Since $J_{k} (x) \le
C_1 x_k$, $\phi_{k,\varnothing}(x)=0$ for $x_k \le \epsilon/3C_1$. Thus,
\begin{equation}
\label{f01}
G_k(x) \;=\; 0 \quad\text{for}\quad x_k \le \epsilon/3C_1\;.
\end{equation}
On the other hand, by \eqref{35}, $J_{D\cup\{k\}} (x) \ge c_1
x_k$. Hence, since $\Phi(r)=1$ for $r\ge 3$ and since $\alpha\le 1$,
\begin{equation}
\label{f03}
G_k(x) \;=\; 1 \quad\text{for}\quad x_k \ge \epsilon /c_1
\alpha^{|A|} \;.
\end{equation}
By similar reasons, we have
\begin{equation}
\label{f02}
\nabla \phi_{k,D} (x) \;=\; 0 \quad\text{if}\quad
3 \alpha^{|D|} c_1 \Vert x\Vert_{D\cup \{k\}} \ge 2\epsilon\;.
\end{equation}
Finally, since $\Phi(r)=0$ for $r\le 0$, by \eqref{35}, 
\begin{equation}
\label{f05}
\phi_{k,D} (x) \;=\; 0 \quad\text{if}\quad
3 \alpha^{|D|} C_1 \Vert x\Vert_{D\cup \{k\}} \le \epsilon\;.
\end{equation}

Let $G: \Sigma \to \bb R$ be defined by
\begin{equation}
\label{f04}
G(x) \;=\; \prod_{k\in B} G_k(x)\;, \quad x\in \Sigma \;.
\end{equation}
We claim that 
\begin{equation}\label{ddh}
H(x) \;=\; F(x)\, G(x)\;,\quad x\in \Sigma \;,
\end{equation}
fulfills all the properties required in the lemma. It follows from
\eqref{f03} that $H$ and $F$ coincide on the set
$\Lambda_{B}(\delta)$, where $\delta = \epsilon /c_1
\alpha^{|A|}$. Since $\epsilon>0$ is arbitrary, this proves the first
assertion of the lemma. By \eqref{f01}, $H$ belongs to $\mc D_B$. It
remains to show that $H$ belongs to $\mc D_A$.

Fix $j\in A$. It is clear that $F_1 F_2$ belongs to $\mc D_j$ if both
functions belong. It is also clear that $\Xi (F)$ belongs to $\mc
D_j$ if $F$ belongs to $\mc D_j$ and if $\Xi : \bb R\to\bb R$ is a
smooth function. Therefore, as $F$ and $I_C$ belong to $\mc D_j$ if
the set $C$ contains $j$, to prove that $H$ belongs to $\mc D_j$ it is
enough to show that all terms which do not contain $x_j$ in the
product in \eqref{f04} belong to $\mc D_j$. A general term in such
product has the form $\phi_{k,D}(x)$, where $D$ is a proper subset $A$
which does not contain $x_j$ and $k\in B$. We will not prove that
$\phi_{k,D}(x)$ belongs to $\mc D_j$, but that $\phi_{k,
  D\cup\{j\}}(x) \,[ \bs v_j \cdot\nabla \phi_{k,D}(x)]$ vanishes for
$x_j$ small enough.  Indeed, by \eqref{f02} and \eqref{f05}, this
product vanishes unless
\begin{equation*}
3 \alpha^n c_1 \Vert x\Vert_{D\cup \{k\}} \;\le\; 2\epsilon \quad
\text{and}\quad 3 \alpha^{n+1} C_1 \Vert x\Vert_{D\cup \{k,j\}} \;\ge\; \epsilon\;,
\end{equation*}
where $|D|=n$. It follows from these inequalities and from the
definition of $\alpha$ that $x_j\ge \epsilon/2$. This completes the
proof of the second assertion of the lemma. 

We turn to the last assertion. Assume that $[F]_B$, introduced in
\eqref{cf03}, has a compact support contained in
$\mathring{\Sigma}_B$. There exists therefore $\epsilon_0>0$ such that
\begin{equation*}
F(x) \;=\; 0 \;, \quad x\in \Sigma_{B,0} \setminus
\Lambda_B(\epsilon_0) \;.
\end{equation*}
Let $\epsilon = \epsilon_0 c_1 \alpha^{|A|}$. By \eqref{f03} and
\eqref{ddh}, $H=F$ on $\Lambda_B(\epsilon/c_1 \alpha^{|A|}) =
\Lambda_B(\epsilon_0)$. On the other hand, since $F$ vanishes on
$\Sigma_{B,0} \setminus \Lambda_B(\epsilon_0)$, by \eqref{ddh}, $H$
also vanishes on this set. This completes the proof of the last
assertion of the lemma.
\end{proof}

\section{Absorption at the boundary}
\label{rec01}

In this section, we prove Theorem \ref{amp0} which states that any
solution of the $\mf L$-martingale problem is absorbed at the
boundary. Throughout this section, $\bb P_x$ denotes a solution of the
$\mf L$-martingale problem starting at $x\in \Sigma$, and $p$ is a
real number satisfying
\begin{equation}
\label{bp1}
1 <p<b<p+1
\end{equation}
This is possible because we assumed $b>1$.

\subsection{First time interval}

Recall the definition of the hitting time $\sigma_1$ introduced in
\eqref{sib}.  As a first step in the proof of Theorem \ref{amp0}, we
show that before $\sigma_1$ the empty sites remain empty.

\begin{proposition}
\label{s01}
Fix $z\in\Sigma$, and let $A=\ms A(z)$, $B=\ms B(z)$. Assume that $A$
is nonempty. Then,
\begin{equation*}
\bb P_z \, \big[ \, \Vert X_{t}\Vert_A = 0 
\,,\, 0 \le t < \sigma_1 \, \big] \;=\; 1\;.
\end{equation*}
\end{proposition}

The proof of Proposition \ref{s01} is divided in several steps.  Fix
$z\in \Sigma$, and let $A=\ms A(z)$ and $B=\ms B(z)$.  Obviously,
$z \in \mathring{\Sigma}_{B,0}$.  Consider the function $f_A:\Sigma\to
\bb R$ given by
\begin{equation}
\label{fa1}
f_A (x) \;=\;  \pi_A(x) ^{p+1} \;.
\end{equation}
As $p>1$, it is clear that $f_A$ belongs to $\mc D_A$. We start
showing that $\mf L f_A$ is negative in a neighborhood of $z$. Denote
by $\Sigma_{B}(\epsilon)$, $\epsilon>0$, the compact neighborhood of
$\Sigma_{B,0}$ defined by:
\begin{equation*}
\Sigma_{B}(\epsilon) \;:=\; \{x\in \Sigma : \max_{j\in A} x_j 
\le \epsilon\}\;.
\end{equation*}

Recall from Section \ref{sec01}.E that we denote by $\mc L^*$ the
adjoint of the generator $\mc L$ in $L^2(\bs m)$, and by $\mc S$ the
symmetric part.

\begin{lemma}
\label{psi4}
There exists $a_0>0$ such that for all $\epsilon>0$ and $x\in
\Sigma_{B}(a_0\epsilon)\cap \Lambda_B(\epsilon)$ we have $\mf L f_A(x)
\le 0$.
\end{lemma}

\begin{proof}
Let $a_0$ be given by
\begin{equation*}
a_0^{-1} \;:=\; \max_{k\in A} \frac{b \,\langle \mc L\bs e_k, 
{\mb 1}\{B\}\rangle_{\bs m}}{(b-p)\, \< (-\mc S) \bs e_k,\bs e_k\>_{\bs
  m}}\;\cdot
\end{equation*} 
Note that the numerator is non-negative for each $k\in A$ because
$(\mc L\bs e_k)(j) = r(j,k) \ge 0$ for $j\in B=A^c$. Furthermore, by
irreducibility, $\langle \mc L\bs e_k, {\mb 1}\{B\}\rangle_{\bs m} =
\sum_{j\in B} m(j) r(j,k)$ is strictly positive at least for one $k\in
A$. This proves that $a_0^{-1} >0$ because $b>p$.

Fix $\epsilon>0$ and $x\in\Sigma_{B}(a_0\epsilon)\cap
\Lambda_B(\epsilon)$. A straightforward calculation yields that for
any $j\in S$,
\begin{equation*}
(\bs v_j \cdot \nabla f_A) (x) \;=\; 
(p+1) \sum_{k\in A} \frac{f_A(x)}{x_k} \mc L \bs e_k(j)\;, 
\end{equation*}
so that ${\bf b}(x)\cdot \nabla f_A(x)$ can be written as
\begin{equation*}
b(p+1)\sum_{j,k\in A} \frac{f_A(x)}{x_jx_k} 
\langle \mc S\bs e_k,\bs e_j \rangle_{\bs m}
\;+\; b(p+1)\sum_{k\in A, j\in B} 
\frac{f_A(x)}{x_jx_k} \, \< \mc L\bs e_k,\bs e_j \>_{\bs m}\;.
\end{equation*}
In the above formula, the indicator $\mb 1\{x_j >0\}$, $j\in A$
(resp. $j\in B$), has been removed because $f_A(x)/x_j \to 0$ as
$x_j\to 0$ (resp. $x$ belongs to $\Lambda_B(\epsilon)$). On the other
hand, a computation, similar to the one carried out to obtain
\eqref{sot2}, shows that the second-order piece of $\mf Lf_A(x)$ can
be written as
\begin{equation*}
\begin{split}
& \sum_{j, k \in S} \partial^2_{x_{j},x_{k}}f_A(x)\, \< (- \mc S) \bs
e_j ,\bs e_k \>_{\bs m} \\ 
&\quad =\; - (p+1)^2 \sum_{\substack{j,k\in A \\ j\not = k}} 
\frac{f_A(x)}{x_jx_k}\, \< \mc S \bs e_j ,\bs e_k \>_{\bs m}
\;+\; p(p+1) \sum_{k\in A} \frac{f_A(x)}{x_k^2} 
\< (- \mc S) \bs e_k ,\bs e_k \>_{\bs m} \;.
\end{split}
\end{equation*}
It follows from the previous calculations that $(p+1)^{-1}\mf L
f_A(x)$ is equal to
\begin{equation}
\label{48}
\begin{split}
& - (p+1-b) \sum_{\substack{j, k \in A \\ k\not = j} } 
\frac {f_A(x)}{x_j x_k} \<  \mc S \bs e_j ,\bs e_k \>_{\bs m}
\;+\;  b \sum_{k \in A, j\in B } \frac {f_A(x)}{x_k x_j} 
\langle \mc L\bs e_k,\bs e_j\rangle_{\bs m} \\ 
& \qquad -\; (b-p) \sum_{k \in A} \frac {f_A(x)} {x_k^2}
\< (- \mc S) \bs e_k ,\bs e_k \>_{\bs m} \;.
\end{split}
\end{equation}
Since $p+1-b>0$ the first term in the above expression is clearly
negative. As $x\in \Lambda_B(\epsilon)$, the second term is bounded
above by
\begin{equation*}
\frac{b}{\epsilon} \sum_{k\in A} \frac{f_A(x)}{x_k} \langle \mc L\bs
e_k,{\mb 1} \{B\}\rangle_{\bs m}\;.
\end{equation*}
Since $b>p$ and $x\in \Sigma_B(a_0\epsilon)$ the last term is bounded above by
\begin{equation*}
- \frac{(b-p)}{a_0\epsilon } \sum_{k \in A} \frac {f_A(x)} {x_k} 
\< (- \mc S) \bs e_k ,\bs e_k \>_{\bs m} \;.
\end{equation*}
By the last two estimates and by definition of $a_0$ we conclude that the
expression in \eqref{48} is negative. 
\end{proof}	

In virtue of Lemma \ref{r01}, for each $\epsilon>0$, there exists a
function in $\mc D_S$, denoted by $H^{\epsilon}_A$, such that
\begin{equation}
\label{psi1}
H^{\epsilon}_A(x) \;=\; f_A(x) \quad \textrm{and} \quad 
\mf LH^{\epsilon}_A(x) \;=\; \mf Lf_A(x) \;,\quad
x\in \Lambda_B(\epsilon)\;.
\end{equation}
For every $\epsilon>0$, denote by $\tau_0(\epsilon)$ the exit time from
the compact set $\Sigma_{B}(a_0\epsilon)\cap \Lambda_B(\epsilon)$:
\begin{equation*}
\tau_0(\epsilon) \;:=\; \inf\{t\ge 0 : X_t \not\in 
\Sigma_{B}(a_0\epsilon)\cap \Lambda_B(\epsilon)\}\;. 
\end{equation*}

\begin{lemma}
\label{s02}
For every $0<\epsilon < \min_{j\in B} z_j$ we have
\begin{equation*}
\bb P_{z} \big[\, \pi_ A (X_{t}) = 0 \;,\;\; 
0\le t\le \tau_{0}(\epsilon) \big] \;=\; 1 \;.
\end{equation*}
\end{lemma}

\begin{proof}
Fix $t>0$. Since $H^{\epsilon}_A\in \mc D_S$ we have
\begin{equation*}
\bb E_{z}[H^{\epsilon}_A(X_{t\land \tau_0(\epsilon)})] \;=\; 
H_A^{\epsilon}(z) \;+\; \bb E_{z} \Big[\int_0^{t\land
  \tau_0(\epsilon)} 
\mf LH^{\epsilon}_A(X_{s})\,ds \Big] \;.
\end{equation*}
By definition of $\tau_0(\epsilon)$, by \eqref{psi1} and by Lemma
\ref{psi4}, the expectation on the right hand side in the above
equation is negative. Hence
\begin{equation*}
\bb E_{z}[H^{\epsilon}_A(X_{t\land \tau_0(\epsilon)})] \;\le\; H_A^{\epsilon}(z) \;.
\end{equation*}
By \eqref{psi1} and since $\epsilon < \min_{j\in B}z_j$ we may replace
$H_A^{\epsilon}$ by $f_A$ in the above inequality. Since $f_A(z)=0$,
after this replacement we have that $\bb E_{z} \big [f_A(X_{t\wedge
  \tau_0 (\epsilon)} ) \big] \le 0$. This proves that for all $t\ge
0$,
\begin{equation*}
\bb P_{z}\big[ \pi_A(X_{t \wedge \tau_0(\epsilon)}) = 0 \big] \;=\; 1 \;.
\end{equation*}
To complete the proof it remains to consider a countable set of times
dense in $\bb R_+$.
\end{proof}

We have thus shown that, under $\bb P_{z}$, before time
$\tau_0(\epsilon)$ at least one of the coordinates in $A$ must be
zero. To improve this result, we consider for each nonempty subset
$D\subseteq A$ the function $f_D:\Sigma_{D\cup B}\to \bb R$ defined as
$f_D(x)=\prod_{j\in D}x_j^{p+1}$ so that the definition of $f_A$ is
consistent with \eqref{fa1}.

At this point, we reduce the neighborhood of $z$ to obtain estimates,
similar to the ones derived in Lemma \ref{psi4}, for all functions
$f_D$. For each nonempty $D\subseteq A$, let $a_D>0$ be given by
\begin{equation*}
a_D^{-1} \;=\; \max_{k\in A} \frac{b \,\langle \mc L^{B\cup D} \bs
  e_k, {\mb 1} \{ B \} \rangle_{\bs m}}
{(b-p)\, \< (- \mc S^{B\cup D}) \bs e_k,\bs e_k\>_{\bs m}}
\end{equation*}
where here $\langle \cdot, \cdot \rangle_{\bs m}$ represents the
$L^2$-inner product with respect to $\bs m$ restricted to $B\cup D$,
$\{\bs e_k : k\in B\cup D\}$ the canonical basis for $\bb R^{B\cup
  D}$, and $\mc S^{B\cup D}$ the symmetric part of the generator $\mc
L^{B\cup D}$ in $L^2(\bs m)$. Since $\bs r^{B\cup D}$ is irreducible
and $b>p$, $a_D$ is well defined and strictly positive for all
nonempty subsets $D$ of $A$.

We set $\bs a_0 \;:=\; \min\{a_D : \varnothing\subsetneq D \subseteq A
\}$ and denote
\begin{equation*}
K_{z}(\epsilon) \;:=\; \Sigma_B(\bs a_0\epsilon)\cap 
\Lambda_B(\epsilon)\;,
\end{equation*}
for all $\epsilon>0$. Of course, $z \in K_{z}(\epsilon)$ if and only
if $\epsilon\in (0,\epsilon_0)$ where
\begin{equation*}
\epsilon_0 \;:=\; \min_{j\in B}z_j>0\;.
\end{equation*}

\begin{lemma}
\label{psi5}
Let $D$ be a nonempty subset of $A$. For all $\epsilon\in
(0,\epsilon_0)$,
\begin{equation*}
\mf L_{D\cup B} f_D(x_{D\cup B})\le 0\;,\quad  x\in K_{z}(\epsilon)\;.
\end{equation*}
\end{lemma}

\begin{proof}
The proof is similar to the one of Lemma \ref{psi4}. One just needs to
replace $\mf L$, $\mc L$ and $\mc S$ by the respective operators $\mf
L_{B\cup D}$, $\mc L^{B\cup D}$ and $\mc S^{B\cup D}$.
\end{proof}

For each nonempty proper subset $D$ of $A$, Lemma \ref{g1} permits to
extend the function $f_D:\Sigma_{D\cup B}\to \bb R$ to a function
$F_D: \Sigma\to \bb R$ which belongs to $\mc D_A$ and such that:
\begin{equation}
\label{fd1}
\begin{split}
& F_D(x)\;=\; f_D(x_{B\cup D})\;=\; \pi_{D}(x)^{p+1}
\;,\quad  x\in \Sigma_{B\cup D,0}\;, \\
&\quad \mf L F_D(x)\;=\; \mf L_{B\cup D}f_D(x_{B\cup D})\;,\quad  
x\in \mathring{\Sigma}_{B\cup D,0}\;, \\
&\qquad \mf LF_D(x)\;=\;0 \quad \textrm{if $\pi_D(x)=0$}\;.
\end{split}
\end{equation}
Moreover, since each $f_D$ is positive, it follows from the
construction presented in Lemma \ref{g1} that $F_D(x)\ge 0$ for all
$x\in \Sigma$. Denote by $H^{\epsilon}_D$ the function in $\mc D_S$
obtained by Lemma \ref{r01} from $F_D$. Thus, for each $\epsilon\in
(0,\epsilon_0)$ and for each nonempty, proper subset $D$ of $A$,
\begin{equation}
\label{hef1}
H^{\epsilon}_D(x) \;=\; F_D(x) \quad\textrm{and} \quad 
\mf L H^{\epsilon}_D(x) \;=\; \mf L F_D(x)\;, \quad 
x\in \Lambda_{B}(\epsilon)\;.
\end{equation}
Since $z\in K_{z}(\epsilon)\cap \Sigma_{D\cup B,0}$, by the first
property in \eqref{fd1}, by \eqref{hef1}, and by the positivity of
$F_D$,
\begin{equation}
\label{hx0}
\begin{split}
& H^{\epsilon}_D(z) \;=\; F_D(z) \;=\; \pi_{D}(z)^{p+1} \;=\; 0 \;, \\
&\quad H^{\epsilon}_D(x) \;=\; F_D(x) \; \ge \; 0 \;,\quad  
x\in \Lambda_B(\epsilon)\;,\\
&\qquad H^{\epsilon}_D(x) \;=\; F_D(x) \; = \; 
\pi_{D}(x)^{p+1} \;, \quad  x\in \Lambda_B(\epsilon)\cap 
\Sigma_{B\cup D,0}\;.  
\end{split}
\end{equation}

\begin{lemma}
\label{ns08}
For all $\epsilon>0$ there exists a constant $C(\epsilon)>0$ such that
for all nonempty, proper subset $D$ of $A$,
\begin{equation*}
\mf L H^{\epsilon}_D(x) \;\le \; C(\epsilon) \mb 1\{\pi_D(x)>0 \,,\, 
\Vert x \Vert_{A\setminus D} > 0\}\;, \quad  x\in K_{z}(\epsilon)\;.
\end{equation*}
\end{lemma}

\begin{proof}
Fix $\epsilon>0$. Since each function $\mf L H^{\epsilon}_D$,
$\varnothing \subsetneq D \subsetneq A$, is continuous on $\Sigma$,
\begin{equation*}
C(\epsilon) \;:=\; \sup \big\{ |\mf L H^{\epsilon}_D(x)| : 
x\in \Sigma \,,\, \varnothing \subsetneq D \subsetneq A\} 
\;<\; \infty \;.
\end{equation*}

Fix a nonempty subset $D$ of $A$ and $x\in K_{z}(\epsilon)$. Since
$x\in \Lambda_B(\epsilon)$, by \eqref{hef1} and by the third property
in \eqref{fd1},
\begin{equation*}
\mf L H^{\epsilon}_D (x) \;=\; \mf L H^{\epsilon}_D(x) 
\, {\mb 1}\{\pi_D(x) >0 \}\;.
\end{equation*}
On the other hand, if $\pi_D(x)>0$ and $\Vert x \Vert_{A\setminus
  D}=0$, $x\in \mathring{\Sigma}_{D\cup B,0}$. Hence, in this case, by
\eqref{hef1} and by the second property in \eqref{fd1},
\begin{equation*}
\mf LH^{\epsilon}_D(x) \;=\; \mf L F_D(x) \;=\; 
\mf L_{D\cup B} f_D(x_{D\cup B}) \;.
\end{equation*}
From this observation and from Lemma \ref{psi5} we conclude that
\begin{equation*}
\textrm{$\mf L H^{\epsilon}_D(x)\le 0$\; if \,$\pi_D(x)>0$ \, and \, 
$\Vert x \Vert_{A\setminus D}=0$} \;.
\end{equation*}
It follows from the previous estimates that
\begin{equation*}
\mf L H^{\epsilon}_D(x) \;\le\; {\mb 1} \big\{ \pi_D(x)>0\,,\, 
\Vert x \Vert_{A\setminus D} >0 \big\} \,\mf L H^{\epsilon}_D(x)\;.
\end{equation*}
The assertion of the lemma is a simple consequence of this inequality
and the definition of $C(\epsilon)$.
\end{proof}

We may now improve Lemma \ref{s02} in the following sense. For every
$\epsilon\in (0,\epsilon_0)$, define $\tau(\epsilon)$ as the exit time
from the compact neighborhood $K_{z}(\epsilon)$ of $z$:
\begin{equation*}
\tau(\epsilon) \;:=\; \inf\{t\ge 0 : X_t \not\in K_{z}(\epsilon)\}\;.
\end{equation*}

\begin{lemma}
\label{47}
For all $\epsilon\in (0,\epsilon_0)$ and nonempty subset $D$ of $A$, 
\begin{equation*}
\bb P_{z}\big[ \pi_D (X_{t}) = 0 \; ,\; 0\le t\le \tau
(\epsilon)  \big] \;=\; 1 \;.
\end{equation*}
\end{lemma}

\begin{proof}
Fix  $\epsilon\in (0,\epsilon_0)$. By Lemma \ref{s02}, the
claim holds for $D=A$. We extend the assertion to all nonempty
$D\subseteq A$ by a recursive argument. Fix $0\le n < |A|-1$, and
assume that, the assertion of the lemma holds for all $D\subseteq A$
with $|D| \ge |A|-n$. Consider a subset $D'\subseteq A$ such that
$|D'|=|A|-n-1$. By the recurrence hypothesis,
\begin{equation}
\label{van}
\bb P_{z} \big[ \pi_{D'}(X_{s\wedge \tau(\epsilon)} )>0 \,,\, 
\Vert X_{s\wedge \tau(\epsilon)} \Vert_{A\setminus D'} > 0  \big] \;=\; 0\;,
\end{equation}
for all $s\ge 0$. Fix $t\ge 0$. Since $H^{\epsilon}_{D'}\in \mc D_S$, 
\begin{equation*}
\bb E_{z} [H^{\epsilon}_{D'}(X_{t\land \tau(\epsilon)})] 
\;=\; H^{\epsilon}_{D'}(z) \;+\; \bb E_{z} \Big[ 
\int_0^{t\land \tau(\epsilon)} \mf L H^{\epsilon}_{D'} (X_s) \, ds \Big] \;.
\end{equation*}
Thus, by the first property in \eqref{hx0} and by Lemma \ref{ns08}, we
get
\begin{equation*}
\begin{split}
& \bb E_{z} [H^{\epsilon}_{D'}(X_{t\land \tau(\epsilon)})] \\ 
&\quad \le \; C(\epsilon) \, \bb E_{z} \Big[ \int_0^{t\land
  \tau(\epsilon)} \mb 1\{\pi_{D'}(X_s)>0 \,,\, 
\Vert X_s \Vert_{A\setminus D'} >0\} \, ds \Big] \;.
\end{split}
\end{equation*}
By \eqref{van}, the right hand side of the previous expression
vanishes. By the last property in \eqref{hx0}, on the set $\{ \Vert
X_{t\wedge \tau(\epsilon)} \Vert_{A\setminus D'} =0\}$,
$H^{\epsilon}_{D'}(X_{t\land \tau(\epsilon)})= \pi_{D'}(X_{t\wedge
  \tau(\epsilon)})^{p+1}$.  Therefore, by the second property of
\eqref{hx0}, 
\begin{equation*}
\bb E_{z} \big[  \mb 1\big\{
\Vert X_{t\wedge \tau(\epsilon)} \Vert_{A\setminus D'} =0 \big\}
\, \pi_{D'}(X_{t\wedge \tau(\epsilon)})^{p+1} \, \big] 
\;\le\; \bb E_{z} \big[ H^{\epsilon}_{D'}(X_{t\wedge \tau(\epsilon)}) \big]
\;=\; 0\;,
\end{equation*}
so that
\begin{equation*}
\bb P_{z} \big[ \, 
\Vert X_{t\wedge \tau(\epsilon)} \Vert_{A\setminus D'} = 0 
\,,\, \pi_{D'}(X_{t\wedge \tau(\epsilon)})>0 \, \big] 
\;=\; 0\;.
\end{equation*}
The previous identity and \eqref{van} yield that
\begin{equation*}
\bb P_{z} \big[ \pi_{D'}(X_{t\wedge \tau(\epsilon)} )>0 \big] \;=\; 0\;.
\end{equation*}
Finally, taking a countable set of times $t$, dense in $\bb R_+$, we
conclude that the assertion of the lemma holds for $D'$. This
completes the proof.
\end{proof}

The previous lemma with $D=\{j\}$, $j\in A$, yields that, for any
$\epsilon\in (0,\epsilon_0)$,
\begin{equation*}
\bb P_{z}\big[ \, \Vert X_{t}\Vert_A = 0 
\;\text{ for all }\; 0\le t\le \tau(\epsilon)  \big] \;=\; 1\;.
\end{equation*}
Since $\tau(\epsilon)$ is the first time $t$ in which either
$\max_{j\in A} X_t(j) > \bs a_0 \epsilon$ or $\min_{j\in B} X_t(j) <
\epsilon $, 
\begin{equation*}
\bb P_{z} \big[\, \Vert X_{t}\Vert_A = 0 \;\text{ for all }\; 
0\le t\le h_B(\epsilon)  \big] \;=\; 1\;,
\end{equation*}
where, $h_B(\epsilon)$ is the exit time of $\Lambda_B(\epsilon)$:
\begin{equation}\label{hbep}
h_B(\epsilon) \;:=\; \inf\{t\ge 0: \min_{j\in B} X_t(j) 
< \epsilon\}\;,\quad \epsilon>0\;.
\end{equation}
To complete the proof of Proposition \ref{s01}, it remains to let
$\epsilon\downarrow 0$.

\subsection{Absorption at the boundary}
We prove that $\bb P_x$ is absorbing for every $x\in \Sigma$. This
result follows from next proposition.

\begin{proposition}
\label{abs2}
For all $x\in \Sigma$, $n\ge 0$,
\begin{equation*}
{\bb P}_x \big[ \sigma_n=\infty \;\; \textrm{or} \;\; 
\ms A_n = \ms A(X_t) \;\; \text{for all} \;\;
t \in [\sigma_n, \sigma_{n+1}) \big] \;=\; 1\;.
\end{equation*}
\end{proposition}

The assertion for $n=0$ has been proved in Proposition \ref{s01}. To
extend this claim to the remaining time intervals we use the concept of
regular conditional probability distributions (r.c.p.d.)  and the
techniques introduced in \cite{sv79}. Given a probability measure $\bb
P$ on $C(\bb R_+,\Sigma)$ and $n\ge 1$, for $\omega\in \{\sigma_n <
\infty\}$, define a set of probability measures $\bb P^n_{\omega}$ on
$C(\bb R_+,\Sigma)$ as follows. First, choose a r.c.p.d. $\{\bb P_{\mf
  \omega}\}$ for $\bb P$ given the $\sigma$-field $\ms
F_{\sigma_n}$. Then, define
\begin{equation}
\label{dcon}
\bb P^{n}_{\omega} \;:=\; \bb P_{\omega} \circ
\theta_{\sigma_n(\omega)}^{-1} 
\;, \quad \textrm{for $\omega \in \{\sigma_n<\infty\}$}\;,
\end{equation}
where we recall that $(\theta_t)_{t\ge 0}$ stands for the semigroup of
time translations. Next lemma is an immediate consequence of Theorem
1.2.10 in \cite{sv79}.

\begin{lemma}
\label{mpcon}
Let $\bb P$ be a solution of the $\mf L$-martingale problem and $n\ge
1$. Given $H\in \mc D_S$ there exists a $\bb P$-null set $\ms N\in \ms
F_{\sigma_n}$ such that, for all $\omega\in \ms N^c\cap
\{\sigma_n<\infty\}$,
\begin{equation*}
H(X_t) - \int_0^t \mf L H(X_s) \, ds \;,\quad t\ge 0
\end{equation*}
is a $\bb P^n_{\omega}$-martingale with respect to $(\ms F_t)_{t\ge 0}$.
\end{lemma}

This lemma permits to employ the arguments presented in the proof of
Proposition \ref{s01} to the general setting of Proposition \ref{abs2}.

\begin{proof}[Proof of Proposition \ref{abs2}]
Fix $x\in \Sigma$ and $n\ge 1$. To keep notation simple denote by $\bb
P^n_{\omega}$ the measure $(\bb P_x)^{n}_{\omega}$ defined by
\eqref{dcon}. By taking the conditional expectation with respect to
$\ms F_{\sigma_n}$ in the probability appearing in the statement of
Proposition \ref{abs2}, we conclude that it is enough to show that
\begin{equation}
\label{cas1}
\bb P^{n}_{\omega} \big\{ \ms A_n(\omega) = \ms A(X_t) 
\,,\, 0\le t < \sigma_1 \big\} \;=\; 1 \;,
\end{equation}
for $\bb P_x$-almost all $\omega \in \{\sigma_n<\infty\}$. By Lemma
\ref{mpcon}, there exists a $\bb P_x$-null set $\ms N\in \ms
F_{\sigma_n}$ such that, for all $\omega \in \ms N^c\cap
\{\sigma_n<\infty\}$, for all nonempty subsets $D$ of $S$ and for a
sequence $\epsilon_k\downarrow 0$,
\begin{equation*}
H^{\epsilon_k}_D(X_t) - \int_0^t \mf L H^{\epsilon_k}_D(X_s) \, ds 
\;,\quad t\ge 0
\end{equation*}
is a $\bb P^n_{\omega}$-martingale with respect to $(\ms F_t)_{t\ge
  0}$, where $H^{\epsilon}_D$ are the functions introduced in the
previous subsection. At this point, we may repeat the argument
presented in the proof of Proposition \ref{s01} to conclude that
\eqref{cas1} holds for all $\omega\in \ms N^c\cap
\{\sigma_n<\infty\}$.
\end{proof}

Theorem \ref{amp0} is a simple consequence of Proposition \ref{abs2}.

\section{Uniqueness}
\label{rec03}

In this section, we prove that for any $x\in \Sigma$ there exists at
most one solution of the $\mf L$-martingale problem starting at
$x$. We start showing that any such solution also solves the
martingale problem determined by $\ms L$ in the form stated in Theorem
\ref{amp1}. Then, we show in Proposition \ref{uni2} that this fact
along with the absorbing property, proved in the previous section,
provides the desired uniqueness.\smallskip

\begin{proof}[Proof of Theorem \ref{amp1}]
Fix $z\in \Sigma$ and a function $F\in D_0(\Sigma)$. Let
\begin{equation*}
M^F_t \;:=\; F(X_t) - \int_0^t  \ms LF(X_s)ds - \int_0^t
F(X_s)dN^S_s\;, \quad t\ge 0\;.
\end{equation*}
Clearly $(M^F_t)_{t\ge 0}$ is $\ms F_t$-adapted and $M^F_t$ is bounded
for each $t\ge 0$. For each proper subset $B$ of $S$, with $|B|\ge 2$,
define $G_B:\Sigma\to \bb R$ as $G_B=[F]_B\circ \Upsilon$, where
$\Upsilon$ is the linear map defined in \eqref{bsu}. By Lemma
\ref{s04}, $G_B$ belongs to $\mc D_A$. Since $[F]_B$ has compact
support contained in $\mathring{\Sigma}_B$, we may apply the last
assertion in Lemma \ref{r01} to each $G_B$. In this way, for each
proper subset $B$ of $S$ with at least two elements, we obtain a
function $H_B\in \mc D_S$ such that
\begin{equation}
\label{hfb1}
\begin{split}
& H_B(x) \;=\; G_B(x) \;=\; F(x)\, \mb 1\{x\in \mathring{\Sigma}_B\}
\;, \quad x\in \Sigma_{B,0}\;, \\
&\quad \mf L H_B(x) \;=\; \mf L G_B(x) \;=\; \mf L_B [F]_B (x_B) 
\;, \quad x\in \Lambda_B(\epsilon)\cap \Sigma_{B,0}\;.  
\end{split}
\end{equation}
By continuity of $\Upsilon$, we may choose $\epsilon>0$ small enough
for $G_B$ to vanish in a neighborhood of $\Sigma_{B,0}\setminus
\Lambda_B(\epsilon)$. For such $\epsilon$,
\begin{equation*}
\mf L H_B(x) \;=\; \mf L G_B(x) \;=\; 0 \;=\; \mf L_B [F]_B (x_B)\;, 
\quad x\in \Sigma_{B,0} \setminus \Lambda_B(\epsilon) \;. 
\end{equation*}
The identity $\mf L H_B(x) = 0$ and $\mf L_B [F]_B (x_B)=0$ are in
force because $H_B$ and $[F]_B$ vanish if $x_j$ is small enough for
some $j\in B$. On the other hand, by its definition, the function
$G_B$ vanishes if $x_j + \Vert x\Vert_A$ is small enough for some
$j\in B$, which explains why $\mf L G_B(x)= 0$.

It follows from the two previous displayed equations that
\begin{equation}
\label{hfb2}
\mf L H_B(x) \;=\;\ms LF(x) \;, \quad x\in \mathring{\Sigma}_{B,0}\;.
\end{equation}
In addition, define $H_S$ as equal to $[F]_S$ (which is equal to $F \,
\mb 1\{\mathring{\Sigma}\}$) and, for all $j\in S$, $H_{\{j\}}$ as a
constant function equal to $F(\bs e_j)$ so that $H_B\in \mc D_S$, for
all nonempty subset $B$ of $S$. Therefore,
\begin{equation}\label{mb1}
M^B_t \;:=\; H_B(X_t) - \int_0^t \mf LH_B(X_s)\,ds\;,\quad t\ge 0\;,
\end{equation}
is a $\bb P_{z}$-martingale with respect to $(\ms F_t)_{t\ge 0}$ for
all $\varnothing \subsetneq B\subseteq S$. 

On the \emph{absorbing event}
\begin{equation*}
\bigcap_{n\ge 0} \{ \textrm{$\ms A_n \subseteq \ms A(X_{t})$ 
for all $t\ge \sigma_n$} \}\;,
\end{equation*}
we have that 
\begin{equation*}
F(X_t) - F(X_0) \;-\; \int_0^t  F(X_s)\, dN_s \;=\;  
F(X_t) \;-\; \sum_{n=0}^{N_t} F(X_{\sigma_n}) \;,\quad t\ge 0\;.
\end{equation*}
For each $n\ge 0$ such that $\sigma_{n+1}<\infty$, $H_{\ms
  B_n}(X_{\sigma_{n+1}})=0$ because, as already observed, $H_B(x)=0$
if one of the coordinates $x_j$, $j\in B$, vanishes. Therefore, by
\eqref{hfb1}, on the absorbing event the right hand side of the
previous expression is equal to
\begin{equation*}
\begin{split}
& \sum_{n=0}^{N_t} \{ H_{\ms B_n}(X_{\sigma_{n+1}\land t}) - F(X_{\sigma_n})\} \\
&\qquad =\; \sum_B\sum_{n=0}^{N_t} \{ H_B(X_{\sigma_{n+1}\land t}) -
H_B(X_{\sigma_n})\} {\mb 1}\{{\ms B_n}=B\}\;,
\end{split}
\end{equation*}
where the first sum on the right hand side is carried over all
nonempty subsets $B$ of $S$. By \eqref{mb1}, for each such $B$, the sum
\begin{equation*}
\sum_{n=0}^{N_t} \{ H_B(X_{\sigma_{n+1}\land t}) - 
H_B(X_{\sigma_n})\} {\mb 1}\{{\ms B_n}=B\}\;,
\end{equation*}
can be written as
\begin{equation*}
\begin{split}
& \sum_{n=0}^{N_t} \Big[ \{ M^B_{\sigma_{n+1}\land t} -
M^B_{\sigma_n}\} + \int_{\sigma_n}^{\sigma_{n+1}\land t} 
\mf L H_B(X_s) ds \Big] {\mb 1}\{{\ms B_n}=B\} \\
&\quad = \;  \sum_{n=0}^{N_t} \{ M^B_{\sigma_{n+1}\land t} -
M^B_{\sigma_n}\} {\mb 1}\{{\ms B_n}=B\} \;+\; 
\int_{0}^t \mf L H_B(X_s) {\mb 1}\{{\ms B}(X_s)=B\} \,ds \;.
\end{split}
\end{equation*}
By \eqref{hfb2}, this last expression equals
\begin{equation*}
\widehat M^B_t \;+\;  \int_{0}^t \ms L F(X_s) 
{\mb 1}\{{\ms B}(X_s)=B\} \,ds
\end{equation*}
where
\begin{equation*}
\widehat M^B_t \;:=\; \sum_{n=0}^{N_t} \{ M^B_{\sigma_{n+1}\land t} 
- M^B_{\sigma_n}\} {\mb 1}\{{\ms B_n}=B\}\;, \quad t\ge 0 \;.
\end{equation*}

Up to this point, we proved that
\begin{equation*}
F(X_t) \;-\; F(X_0) \;-\; \int_0^{t} F(X_s)\,dN_s \;=\;  
\int_{0}^t \ms L F(X_s) \,ds \;+\; \sum_B \widehat M^B_t  \;, 
\quad \textrm{$\bb P_{z}$-a. s.},
\end{equation*}
for all $t\ge 0$. Therefore, it remains to prove that
\begin{equation}
\label{uu}
\bb E_{z} \big[\, \widehat M^B_{t} - \widehat M^B_{s}  
\,|\, \ms F_s \big] \;=\; 0 \;,
\end{equation}
for every $0\le s < t$ and nonempty subset $B$ of $S$. Fix $0\le s<t$,
$\varnothing \subsetneq B\subseteq S$ and $\ms U \in \ms F_s$. By
definition, 
\begin{equation*}
\bb E_z \big[ \{ \widehat M^B_{t} - \widehat M^B_{s} \} 
{\mb 1}\{\ms U\} \big] \;=\; \sum_{n=0}^{|S|} \bb E_z 
\Big[ \big\{ M^B_{(\sigma_{n+1}\land t)\lor s} - 
M^B_{(\sigma_n\land t)\lor s} \big\} {\mb 1} \{\ms U \,,\, \ms B_n=B\} \Big]\;.
\end{equation*}
For each $0\le n \le |S|$,
\begin{equation*}
\begin{split}
& \bb E_z \Big[ \big\{ M^B_{(\sigma_{n+1}\land t)\lor s} - 
M^B_{(\sigma_n\land t)\lor s} \big\}\, {\mb 1}\{\ms U \,,\, \ms B_n=B\} \Big] \\ 
&\quad =\; \bb E_z \Big[ \big\{ M^B_{(\sigma_{n+1}\land t)\lor s} - 
M^B_{(\sigma_n\land t)\lor s} \big\} \, {\mb 1}\{\ms U \,,\, \ms B_n=B
\,,\, \sigma_n \le t \} \Big] \;=\; 0 \;,
\end{split}
\end{equation*}
where last equality follows from the martingale property of
$(M^B_t)_{t\ge 0}$ and from the fact that
\begin{equation*}
\ms U\cap \{\ms B_n=B\} \cap \{\sigma_n \le t\} \;\in\; 
\ms F_{(\sigma_n\land t)\lor s} \;.
\end{equation*}
This proves \eqref{uu} and completes the proof of the theorem.
\end{proof}

A probability measure $\bb P$ on $C(\bb R_+,\Sigma)$ is said to be an
absorbing solution of the $\ms L$-martingale problem if $\bb P$ is
absorbing and for all $F\in D_0(\Sigma)$,
\begin{equation*}
F(X_t) - \int_0^t  \ms LF(X_s)ds - \int_0^t F(X_s)dN^S_s\;, \quad t\ge 0\;,
\end{equation*}
is a $\bb P$-martingale with respect to $(\ms F_t)_{t\ge 0}$. 

\begin{proposition}
\label{uni2}
For each $x\in \Sigma$, there exists at most one absorbing solution of
the $\ms L$-martingale problem starting at $x$.
\end{proposition}

We first prove, in Lemma \ref{unle} below, uniqueness for the absorbing
solution on the time interval $[\sigma_0, \sigma_1)$. For each
$B\subseteq S$ with at least two elements, let
\begin{equation*}
{\bf b}^B_{\epsilon}(x)\;=\; b \sum_{j\in B} 
\frac{m_j}{\epsilon \lor x_j} \,  \bs v^B_j\;.
\quad x\in \bb R^B\;,
\end{equation*}
Thus, for every $\epsilon>0$, ${\bf b}_{\epsilon}^B:\bb R^B\to \bb
R^B$ is a bounded, continuous vector field which coincides with ${\bf
  b}^B$ on
\begin{equation*}
\Lambda_{B,\epsilon} \;:=\; \{ x\in \Sigma_B : \min_{j\in B} x_j 
\ge \epsilon \}\;,
\end{equation*}

Let ${\bf a}^B$ be the matrix whose entries $({\bf a}^B(j,k))_{j,k\in
  B}$ are given by 
\begin{equation*}
{\bf a}^B_{j,k} \;:=\; \langle {\bs e}_j , 
- \mc L^B {\bs e}_k\rangle_{\bs m}\;,\quad j,k \in B\;,
\end{equation*}
where, by abuse of notation, $\{\bs e_j : j\in B\}$ and $\langle \cdot
, \cdot \rangle_{\bs m}$ represent the canonical basis of $\bb R^B$
and the scalar product with respect to $\bs m$ restricted to $B$,
respectively. Consider the symmetric matrix ${\bf a}^B_{\rm
  s}=(1/2)({\bf a}^B + ({\bf a}^B)^{\rm t})$ so that,
\begin{equation*}
\sum_{j,k\in B} v_j \,{\bf a}^B_{\rm s}(j,k)\, w_k 
\;=\; \< v, (-\mc S^B) w\>_{\bs m} \;, \quad v,w\in \bb R^B \;.
\end{equation*}
Note that for all $\epsilon>0$ and $F\in C^2(\Sigma_B)$, 
\begin{equation}\label{slf}
\mf L_B F (x) \;=\;{\bf b}^B_{\epsilon}(x)\cdot \nabla F(x) 
\;+\; \textrm{Tr}\big[{\bf a}^B_{\rm s}\times \textrm{Hess}\,F(x)\big]
\;,\quad x\in \Lambda_{B,\epsilon}\;.
\end{equation}

Let $(X^B_t)_{t\ge 0}$ be the coordinate maps in the path space $C(\bb
R_+,\bb R^B)$, let $\ms F^B_t:=\sigma(X_s : 0\le s\le t)$, $t\ge 0$,
and denote by $h_{B,\epsilon}$ the exit time from
$\Lambda_{B,\epsilon}$:
\begin{equation*}
h_{B,\epsilon} \;:=\; \inf\{ t \ge 0 : X^B_t \not \in
\Lambda_{B,\epsilon} \}\;, \quad \epsilon>0\;.
\end{equation*}
Denote by $\bb Q^{B,\epsilon}_x$, $x\in \bb R^B$, $\epsilon>0$, the
unique solution of the $(\bs b^B_{\epsilon},\bs a^B)$-martingale
problem starting at $x$. Namely, $\bb Q^{B,\epsilon}_x$ is the unique
probability measure on $C(\bb R_+,\bb R^B)$ such that $\bb
Q^{B,\epsilon}_x\{X^B_0=x\}=1$ and such that for all $H:\bb R^B\to \bb
R$ of class $C^2$ and of compact support,
\begin{equation*}
H(X^B_t) - \int_0^t \Big\{ {\bf b}^B_{\epsilon}(X^B_s)\cdot \nabla H(X^B_s) 
\;+\; \textrm{Tr}\big[{\bf a}^B_{\rm s}\times
\textrm{Hess}\,H(X^B_s)\big] \Big\} \,ds\;,\quad t\ge 0
\end{equation*}
is a $\bb Q^{B,\epsilon}_x$-martingale with respect to $(\ms
F^B_t)_{t\ge 0}$. 

Since ${\bs r}^B$ is irreducible, for all $v$ in $\bb R^{B}\setminus
\{0\}$, 
\begin{equation*}
\sum_{j,k\in B} v_j \,{\bf a}^B_{\rm s}(j,k)\, v_k 
\;=\; \< v, (-\mc S^B) v\>_{\bs m} \;>\; 0 \;.
\end{equation*}
Uniqueness of $\{\bb Q^{B,\epsilon}_x\}$ follows from Theorem 7.1.9
in \cite{sv79} and from the previous equation. It also follows from
this theorem that $\{\bb Q^{B,\epsilon}_x : x\in \bb R^B\}$ is Feller
continuous.

The next result asserts that, in the time interval $[0,
h_{B,\epsilon})$, an absorbing solution of the $\ms L$-martingale
problem starting at $x$ coincides with $\bb Q^{B,\epsilon}_x$.

\begin{lemma}
\label{unle}
Fix $x\in \Sigma$, and let $B=\ms B(x)$. Denote by $\bb P_x$ an
absorbing solution of the $\ms L$-martingale problem starting at $x$,
and by $\bb P_x^B$ the law on $C(\bb R_+,\bb R^B)$ of the path
\begin{equation*}
(X_{t}(j), j\in B) \;,\quad t\ge 0
\end{equation*}
under $\bb P_x$. Then, for all $\epsilon>0$, $\bb P_x^B \equiv \bb
Q^{B,\epsilon}_{x_B}$ on $\ms F^B_{h_{B,\epsilon}}$. In particular, if
$\bb P_1$ and $\bb P_2$ are two absorbing solutions of the $\ms
L$-martingale problem starting at $x$, then $\bb P_1\equiv \bb P_2$ on
$\ms F_{\sigma_1}$.
\end{lemma}

\begin{proof}
Fix a starting point $z\in \Sigma$ and let $B=\ms B(z)$. Denote by
$\bb P_z$ an absorbing solution of the $\ms L$-martingale problem
starting at $z$. Fix $\epsilon>0$, and $H:\bb R^B\to \bb R$ of class
$C^2$ and of compact support.  Let
\begin{equation*}
M^{H,\epsilon}_t\;:=\; H(X^B_t) - \int_0^t \Big\{ {\bf
  b}^B_{\epsilon}(X^B_s)\cdot \nabla H(X^B_s) \;+\; 
\textrm{Tr}({\bf a}_{\rm s}^B\times \textrm{Hess}\, H(X^B_s))\Big\}
\,ds\;,\quad t\ge 0\;.
\end{equation*}

It is easy to obtain a function $F\in D_0(\Sigma)$ such that 
\begin{equation*}
F(x) \;=\; H(x_B) \quad \textrm{for all $x\in \Lambda_B(\epsilon/2)
\cap \Sigma_{B,0}$}\;.
\end{equation*}
Recall the definition of $h_B(\epsilon)$ introduced in \eqref{hbep}. By
assumption,
\begin{equation*}
F(X_{t\land h_B(\epsilon)}) - \int_0^{t\land h_B({\epsilon})} 
\ms L F(X_s)\,ds\;, \quad t\ge 0\;,
\end{equation*}
is a $\bb P_{z}$-martingale with respect to $(\ms F_t)_{t\ge 0}$. By
definition of $h_B(\epsilon)$ and $\ms L$, the last two identities
yield that
\begin{equation*}
H(X^B_{t\land h_{B,\epsilon}}) - \int_0^{t\land h_{B,\epsilon}} 
\mf L_B H(X^B_s)\,ds\;, \quad t\ge 0\;,
\end{equation*}
is a $\bb P^B_{z}$-martingale with respect to $(\ms F^B_t)_{t\ge
  0}$. Therefore, by \eqref{slf} and by definition of
$h_{B,\epsilon}$,
\begin{equation}
\label{mm1}
\textrm{$(M^{H,\epsilon}_{t\land h_{B,\epsilon}})_{t\ge 0}$ 
is a $\bb P^B_{z}$-martingale with respect to $(\ms F^B_t)_{t\ge 0}$}\;.
\end{equation}

We now join $\bb P^B_{z}$ and $\{\bb Q^{B,\epsilon}_{z}\}$ at time
$h_{B,\epsilon}$ as follows. To keep notation simple let
$h:=h_{B,\epsilon}$ and let
\begin{equation*}
X^B_{h}(\omega) \;:= \; X^B_{h(\omega)}(\omega)\;, 
\quad \omega \in \{ h_{B,\epsilon}<\infty\}\;.
\end{equation*}
Since $\{\bb Q^{B,\epsilon}_x\}$ is Feller continuous, it follows from
Theorem 6.1.2 in \cite{sv79} that there exists a probability $\bb Q$
on $C(\bb R_+, \bb R^B)$ satisfying the following two properties:
\begin{itemize}
\item[(i)] $\bb Q$ coincides with $\bb P^B_{z}$ on $\ms F^B_{h}$.
\item[(ii)] For any $\{\bb Q_{\omega}\}$, a r.c.p.d for $\bb Q$ given
  $\ms F^B_{h}$, there exists a $\bb Q$-null set $\ms N \in \ms
  F^B_{h}$ such that
\begin{equation*}
\bb Q_{\omega} \circ \theta_{h(\omega)}^{-1} \;=\; 
\bb Q^{B,\epsilon}_{X^B_{h}(\omega)} \;, \quad 
\textrm{for all $\omega \in \ms N^c\cap\{h<\infty\}$}\;.
\end{equation*}
\end{itemize}
Note that in Theorem 6.1.2 of \cite{sv79}, $h$ is assumed to be
finite. Nevertheless, the proof of this theorem can easily be adapted
for a general stopping time.

By definition of $\{\bb Q^{B,\epsilon}_x\}$ and by (ii), the process
$(M^{H,\epsilon}_t)_{t\ge 0}$ is a $\bb Q_{\omega} \circ
\theta_{h(\omega)}^{-1}$-martingale, with respect to $(\ms
F^B_t)_{t\ge 0}$, for all $\omega\in \ms N^c \cap \{h<\infty\}$. By
Theorem 1.2.10 in \cite{sv79}, this fact along with \eqref{mm1} and
item (i) above allows us to conclude that $(M^{H,\epsilon}_t)_{t\ge
  0}$ is a $\bb Q$-martingale. We have thus proved that $\bb Q$ is a
solution of the $({\bf b}^B_{\epsilon},{\bf a}_{\rm s}^B)$-martingale
problem. Since $\bb Q\{ X^B_0= z_{B} \}=1$, by uniqueness, $\bb Q =
\bb Q^{B,\epsilon}_{z_B}$. This fact and item (i) completes the proof
of the first assertion of the lemma. 

The second assertion follows from the absorbing property and from the
first assertion by letting $\epsilon\downarrow 0$.
\end{proof}

Given a probability measure $\bb P$ on $C(\bb R_+,\Sigma)$, recall
from \eqref{dcon} the definition of the measure $\bb P^n_{\omega}$,
$\omega \in \{\sigma_n<\infty\}$, $n\ge 1$. We use the probability
measures $\bb P^1_{\omega}$, $\omega \in \{\sigma_1<\infty\}$, to
conclude the proof of the uniqueness stated in Proposition
\ref{uni2}. The proof of next lemma follows the same argument as in
the proof of Theorem 6.1.3 in \cite{sv79}.

\begin{lemma}
\label{svl}
Let $\bb P$ be an absorbing solution of the $\ms L$-martingale problem.
Then, there exists a $\bb P$-null set $\ms N\in \ms F_{\sigma_1}$ such
that, for all $\omega \in \ms N^c \cap \{\sigma_1<\infty\}$, $\bb
P^1_{\omega}$ is an absorbing solution of the $\ms
L$-martingale problem starting at $X_{\sigma_1}(\omega)$.
\end{lemma}

\begin{proof}
Let $\Theta$ be a countable subset of $D_0(\Sigma)$ satisfying the
following property: for all $F\in D_0(\Sigma)$, there exists a
sequence $(F_n)_{n\ge 1}$ in $\Theta$ such that
\begin{equation*}
\lim_{n\to\infty} \sup_{x\in \Sigma}\big\{ |F_n(x) - F(x)| \;+\; 
|\ms LF_n(x) - \ms LF(x)| \big\} \;=\; 0\;.
\end{equation*}

By Theorem 1.2.10 of \cite{sv79}, there exists a $\bb P$-null set $\ms
N\in \ms F_{\sigma_1}$ such that, for all $\omega \in \mc N^c\cap \{
\sigma_1 <\infty \}$ and for all $F\in \Theta$,
\begin{equation*}
F(X_t) \;-\; \int_0^t \ms L F(X_s) \, ds \; - \; 
\int_0^t F(X_s) \, d N^S_s  \;,\quad t\ge 0\;,
\end{equation*}
is a $\bb P^1_{\omega}$-martingale. Approximating a function $F$ in
$D_0(\Sigma)$ by a sequence in $\Theta$, we may conclude that the
previous expression is also a $\bb P^1_{\omega}$-martingale for all
$F\in D_0(\Sigma)$. 

Finally, it is easy to see that the $\bb P$-null set $\ms N\in \ms
F_{\sigma_1}$ may be chosen so that $\bb P^1_{\omega}$ is absorbing
for all $\omega\in \ms N^c \cap \{\sigma_1 <\infty\}$.
\end{proof}

We are now in position to complete the proof of uniqueness of
absorbing solutions of the $\ms L$-martingale problem.

\begin{proof}[Proof of Proposition \ref{uni2}]
If the starting point belongs to $\{{\bs e}_j: j\in S\}$, then the
claim is simple consequence of the absorbing property. 

We proceed by induction. Suppose that the claim is valid for any
starting point in $\{x\in \Sigma:|\ms B(x)|\le n\}$ for some $1\le
n<|S|$. Fix some $z\in \Sigma$ such that $|\ms B(z)|=n+1$ and let $\bb
P_{i}$, $i=1,2$, be two absorbing solutions of the $\ms L$-martingale
problem starting at $z$. By Lemma \ref{unle},
\begin{equation}
\label{tfs}
\bb P_1 \;\equiv\; \bb P_2 \quad  {\rm on} \quad \ms F_{\sigma_1} \;. 
\end{equation}
From the absorbing property it follows that
\begin{equation*}
\bb P_i[\sigma_1<\infty \;\; {\rm and}\;\; |\ms B_1|> n]
\;=\; 0\; , \quad i=1,2 \;.
\end{equation*}
By Lemma \ref{svl}, for $i=1,2$, there exists a $\bb P_i$-null set
$\ms N_i\in \ms F_{\sigma_1}$ such that, for all $\omega\in \ms
N_i^c\cap\{\sigma_1<\infty\}$, $(\bb P_i)^{\sigma}_{\omega}$ is an
absorbing solution of the $\ms L$-martingale problem starting at
$X_{\sigma_1}(\omega)$. Take
\begin{equation*}
\ms N \;:=\; \ms N_1\cup \ms N_2\cup 
\{\sigma_1<\infty \; {\rm and}\; |\ms B_1|> n\}\;.
\end{equation*}
It follows from the previous displayed equations that $\bb P_2(\ms
N)=\bb P_1(\ms N)=0$. Fix an arbitrary $\omega \in \ms N^c \cap
\{\sigma_1<\infty\}$. On the one hand, $(\bb P_1)^{\sigma}_{\omega}$
and $(\bb P_2)^{\sigma}_{\omega}$ are absorbing solution of the $\ms
L$-martingale problem starting at $X_{\sigma_1}(\omega)$. On the other
hand, by definition of $\ms N$, $X_{\sigma_1}(\omega)$ belongs to
$\{x\in \Sigma: |\ms B(x)|\le n\}$. Hence, by the inductive
hypothesis, $(\bb P_1)^{\sigma}_{\omega}=(\bb
P_2)^{\sigma}_{\omega}$. The assertion of the proposition follows from
this fact and from \eqref{tfs}.
\end{proof}

\begin{proof}[Proof of Theorem \ref{r03}.]
Theorem \ref{r03} follows from Theorem \ref{amp1} and Proposition \ref{uni2}.
\end{proof}

\begin{proof}[Proof of Proposition \ref{sf01}.]
Fix $x$ in $\Sigma$ and assume that $\ms A(x)=\{j\in S : x_j=0\} \not
= \varnothing$. Let $B=\ms A(x)^c$. It is clear that the measure $\bb
P^{B}_x$ starts at $x_B$ and that it is absorbing. By the proof of
Theorem \ref{amp1}, it solves the $\ms L$-martingale problem
restricted to $\Sigma_B$: for all functions $F\in D_0(\Sigma_B)$,
\begin{equation*}
M^F_t \;:=\; F(X_t) - \int_0^t  \ms LF(X_s)ds - \int_0^t
F(X_s)dN^S_s\;, \quad t\ge 0\;.
\end{equation*}
is a $\bb P^{B}_x$-martingale. The assertion of the proposition
follows now from the uniqueness stated in Proposition \ref{uni2}.
\end{proof}

\section{Existence}
\label{sec02}

In this section we prove the convergence stated in Theorem
\ref{mt2}. This result also guarantees the existence of solutions of
the $\mf L$-martingale problem. By abuse of notation, in this section,
we also denote by $(X_t)_{t\ge 0}$ the coordinate process defined on
$D(\bb R_+,\Sigma)$.

\subsection{Tightness}
We prove in this section that, for any sequence $x_N\in \Sigma_N$,
$N\ge 1$, the sequence of probability measures $\bb P^N_{x_N}$, $N\ge
1$ is tight. Furthermore, we prove that every limit point of the
sequence is concentrated on continuous trajectories. The proof of
tightness is divided in several lemmas.  We start with an estimate
relating the sequence of operators $\mf L_N$, $N\ge 1$ and the
operator $\mf L$. For $j\in S$ and $H \in C^2(\Sigma)$, recall
the notation
\begin{equation*}
(\Delta_j H) (x) \;:=\; \sum_{k\in S} r(j,k) \,
(\partial_{x_k} - \partial_{x_j})^2 H (x) \;.  
\end{equation*}

\begin{lemma}
\label{texp}
If $H$ belongs to $\mc D_S$,
\begin{equation*}
\lim_{N\to\infty} \max_{x\in \Sigma_N} \Big| \, (\mf L_N H) (x) \;-\;
(\mf L H) (x) \;-\; \frac{1}{2} \sum_{j\in S} \{g_j (N x_j) - m_j\}
(\Delta_j H) (x) \, \Big| \; = \; 0 \;.
\end{equation*}
In particular, there exists a finite constant $C_0>0$, which depends
on $H$, such that
\begin{equation*}
\sup_{N\ge 1} \max_{x\in \Sigma_N} 
\big| \, (\mf L_N H) (x)  \, \big| \; \le \; C_0 \;.
\end{equation*}
\end{lemma}

\begin{proof}
Fix a function $H\in \mc D_S$. In view of \eqref{tay1}, by Taylor
expansion, for any $x\in \Sigma_N$,
\begin{equation}
\label{02}
(\mf L_N H) (x)\;=\;
N \sum_{j\in S} g_j (Nx_j)\, [\bs v_j \cdot \nabla H(x)] \;+\; 
\frac{1}{2}\sum_{j\in S} g_j(Nx_j) \, (\Delta_j H) (x) \;+\; R_N \;,
\end{equation}
where $\lim_{N\to\infty}\max_{x\in \Sigma_N}|R_N| = 0$. Since
$g_j(0)=0$, we may introduce the indicator $\mb 1\{x_j>0\}$ in the
first sum and write it as
\begin{equation}
\label{41}
N \sum_{j\in S} \Big\{ \frac{g_j (Nx_j)}{m_j} -1 \Big\}
\mb 1\{x_j>0\} \, m_j \, [\bs v_j \cdot \nabla H(x)] \;+\; 
N \sum_{j\in S} \mb 1\{x_j>0\} \, m_j \, [\bs v_j \cdot \nabla
H(x)] \;. 
\end{equation}
Since $H$ belongs to $\mc D_S$, $\bs v_j \cdot \nabla H(x)=0$ for
$x_j=0$. We may therefore remove the indicator in the second
sum. Since $\bs m$ is an invariant measure for $\bs r$, $\sum_{j\in
  S}m_j{\bs v}_j=0$. By these last observations, the second sum in
\eqref{41} vanishes. The first term in \eqref{02} is thus equal to
\begin{equation}
\label{42}
\begin{split}
& \sum_{j\in S} \Big\{ N x_j \Big[ \frac{g_j (Nx_j)}{m_j} -1 \Big] -b \Big\}
\, \mb 1\{x_j>0\} \, \frac{m_j}{x_j} \, [\bs v_j \cdot \nabla H(x)] \\
&\quad \;+\; b \sum_{j\in S} \mb 1\{x_j>0\} \, \frac{m_j}{x_j} \, [\bs v_j
\cdot \nabla H(x)]\;. 
\end{split}
\end{equation}
The second term in \eqref{42} is $\bs b(x) \cdot \nabla H(x)$, while
the first term is uniformly small in view of \eqref{propg} and because
$H$ belongs to $\mc D_S$. This completes the proof of the first
assertion. The second assertion follows from the first one and from
the fact that $\mf L H$ is continuous on the compact $\Sigma$.
\end{proof}

We start our route to the proof of Proposition \ref{proptight} below
by showing that tightness follows from an estimate, uniform over the
initial position, of the evolution of the process in small time
intervals.

\begin{lemma}
\label{prolt}
Fix a sequence $x_N\in\Sigma_N$, $N\ge 1$. The sequence of probability
measures $\bb P^N_{x_N}$ is tight if for any $\epsilon>0$ and for any
sequence $(N(k), t_{N(k)}, y_{N(k)})$ such that $N(k)\to\infty$,
$y_{N(k)} \to y$ for some $y\in \Sigma$, $t_{N(k)} \to 0$,
\begin{equation*}
\lim_{k\to\infty} \bb P_{y_{N(k)}}^{\scs N(k)} \big[ \, \Vert X_{t_{N(k)}} 
- y_{N(k)} \Vert \ge \epsilon \, \big] \;=\; 0\;.
\end{equation*}
\end{lemma}

\begin{proof}
Fix a sequence $x_N\in\Sigma_N$, $N\ge 1$.
By Aldous criterion, since $\Sigma$ is a compact space, to prove that
the sequence $\bb P^{\scs N}_{x_N}$ is tight, it is enough to show
that for every $T>0$, $\epsilon>0$,
\begin{equation*}
\lim_{\delta\to 0} \limsup_{N\to \infty} \sup_{0\le t\le \delta} 
\sup_{\tau} {\bb P}^{\scs N}_{x_N}
\big[ \, \Vert X_{\tau + t} - X_{\tau}\Vert \ge
\epsilon \, \big] \;=\; 0\;,
\end{equation*}
where the supremum is carried over all stopping times $\tau$ bounded
by $T$. By the strong Markov property, to prove tightness it is
therefore enough to show that for any $\epsilon>0$,
\begin{equation}
\label{04}
\lim_{\delta\to 0} \limsup_{N\to\infty} \sup_{0\le t \le \delta}
\max_{x\in \Sigma_N} \bb P_{x}^{\scs N} \big[\, \Vert X_t - x\Vert
\,\ge\, \epsilon \, \big] \;=\; 0 \;.
\end{equation}
For each $\delta>0$, there exists $t_N = t_N(\delta)\in [0,\delta]$ and
$y_N = y_N(\delta)\in \Sigma_N$, $N\ge 1$, such that
\begin{equation*}
\limsup_{N\to\infty} \sup_{0\le t\le \delta} \max_{x\in \Sigma_N} 
\bb P_{x}^{\scs N} \big[ \, \Vert X_{t} - x \Vert \ge \epsilon \, \big]
\;=\; \lim_{N\to\infty} \bb P_{y_N}^{\scs N} \big[ \, \Vert X_{t_N} 
- y_N \Vert \ge \epsilon \, \big] \;.  
\end{equation*}
On the right hand side the sequences $t_N$ and $y_N$ depend on
$\delta$, $t_N = t_N(\delta)$, $y_N = y_N(\delta)$. We may choose a
further subsequence $\{N(k) : k\ge 1\}$ such that $\lim_k N(k) =
\infty$, $t_{N(k)} \in [0,1/k]$, and
\begin{equation*}
\lim_{\delta\to 0} \lim_{N\to\infty} \bb P_{y_N}^{\scs N} \big[ \, \Vert X_{t_N} 
- y_N \Vert \ge \epsilon \, \big] \;=\;
\lim_{k\to\infty} \bb P_{y_{N(k)}}^{\scs N(k)} \big[ \, \Vert X_{t_{N(k)}} 
- y_{N(k)} \Vert \ge \epsilon \, \big]\;.
\end{equation*}
Since $\Sigma$ is compact we may assume that $\lim_k y_{N(k)} = y \in
\Sigma$. Therefore, if we are able to prove that for any $\epsilon>0$,
and any sequence $(N(k), t_{N(k)}, y_{N(k)})$ such that
$N(k)\to\infty$, $y_{N(k)} \to y \in \Sigma$, $t_{N(k)} \to 0$,
\begin{equation*}
\lim_{k\to\infty} \bb P_{y_{N(k)}}^{\scs N(k)} \big[ \, \Vert X_{t_{N(k)}} 
- y_{N(k)} \Vert \ge \epsilon \, \big] \;=\; 0\;,
\end{equation*}
\eqref{04} holds, and hence the sequence $\bb P^N_{x_N}$ is tight. This
is the assertion of the lemma.
\end{proof}

Denote by $\tau_{\delta}$, $\delta>0$, the first time the process is
at distance $\delta$ from its original position: $\tau_{\delta}= \inf
\{ t\ge 0 : \Vert X_t - X_0 \Vert >\delta \}$. Let $X^\delta$ be the
process $X_t$ stopped at $\tau_{\delta}$:
\begin{equation*}
X^{\delta}_t \;:=\; X_{t\land \tau_\delta} \;. 
\end{equation*}

\begin{lemma}
\label{le3}
Let $x_N\in\Sigma_N$ and $t_N>0$, $N\ge 1$, be sequences such that
$x_N \to x\in \Sigma$ and $t_N\to 0$. Let $B:=\ms B(x)$, $A:=\ms A(x)$
and let $\delta>0$ be such that $\min_{j\in B} x(j) \ge 2
\delta$. Then, for every $\epsilon>0$ sufficiently small, and every
function $F$ in $\mc D_A$,
\begin{equation*}
\lim_{N\to \infty} \bb P^{\scs N}_{x_{N}}\big[\, 
|F(X^\delta_{t_N}) - F(x)| \;\ge\; \epsilon \big] \;=\; 0 \;.
\end{equation*}
\end{lemma}

\begin{proof}
For $N$ sufficiently large and $s\le \tau_\delta$, $X^N_s\in
\Lambda_B(\delta)$. Therefore, by Lemma \ref{r01}, there exists a
function $H\in \mc D_S$ such that
\begin{equation}\label{fxh}
\bb P^{\scs N}_{x_{N}}\big[\, |F(X^\delta_{t_N}) - F(x)| \;\ge\;
\epsilon \big]
\;=\; \bb P^{\scs N}_{x_{N}}\big[\, |H(X^\delta_{t_N}) - H(x)| 
\;\ge\; \epsilon \big]  \;.
\end{equation}

Consider the $\bb P^N_{x_N}$-martingale
\begin{equation*}
M^{N}_t \;=\; H(X_{t}) \;-\; H(x) \;-\; \int_0^t (\mf L_{\scs N}
H)(X_s) \, ds\;, \quad t\ge 0 \;.
\end{equation*}
The probability appearing on the right hand side of \eqref{fxh} is
bounded above by
\begin{equation*}
\bb P^{\scs N}_{x_{N}}\Big[\, \Big| \int_0^{t_N \land \tau_\delta} (\mf L_N
H) (X_{s}) \, ds  \Big| \,\ge\, \epsilon/2 \,\Big] \;+\; 
\bb P^{\scs N}_{x_{N}} \big[\, |M^{N}_{t_N \land \tau_\delta}| \;\ge\;
\epsilon/2 \,\big]\;. 
\end{equation*}
Therefore, since $H$ belongs to $\mc D_S$, by the last assertion of
Lemma \ref{texp}, the time-integral appearing in the first term is
absolutely bounded by $C_0 t_N$ for some finite constant $C_0$ which
depends on $\delta$ and $H$. This proves that the first term in the
previous displayed equation vanishes as $N\uparrow\infty$ because
$t_N\downarrow 0$. By Tchebychef inequality, the second term is
bounded by
\begin{equation*}
\frac 4{\epsilon^2} \, \bb E^{\scs N}_{x_{N}} [
(M^{N}_{t_N \land \tau_\delta})^2] \;=\; 
\frac 4{\epsilon^2} \, \bb E^{\scs N}_{x_{N}} \Big[ 
\int_0^{t_N \land \tau_\delta} \big(\mf L_N{H}^2 - 2{H} \mf L_N
{H}\big)(X_{s}) \, ds \Big]\;.
\end{equation*}
An elementary computation shows that $( \mf L_N {H}^2 - 2 H \mf L_N
{H})(x)$ is absolutely bounded by a finite constant which depends on
$H$, uniformly on $\Sigma_N$. This completes the proof of the lemma.
\end{proof}

\begin{corollary}
\label{le1}
Under the assumptions of Lemma \ref{le3}, for every $\epsilon>0$,
\begin{equation*}
\lim_{N\to\infty} \bb P^{\scs N}_{x_{N}} \big[\, 
\Vert X^\delta_{t_N}\Vert_A \;\ge\; \epsilon \,\big] 
\;=\; 0\;.  
\end{equation*}
\end{corollary}

\begin{proof}
To estimate the first term on the right hand side, let $\Phi :
[0,\infty) \to [0,\infty)$ be a smooth, non-decreasing function such
that $\Phi(r)=0$ for $0\le r\le 1/2$, $\Phi(r)=r$ for $r\ge 1$. For
$\epsilon>0$, let $\Phi_\epsilon (r) = \epsilon\, \Phi(r/\epsilon)$,
and recall the definition of the function $J_A$ given in \eqref{ja1}. Let
\begin{equation*}
F_\epsilon(x)\;:=\; c_1 \, \Phi_\epsilon(J_A(x)/c_1) \; , \quad x\in \Sigma \;.
\end{equation*}
It is easy to check that $F_\epsilon$
belongs to $\mc D_A$ and that $F_\epsilon(x) = J_A(x) \ge c_1 \Vert
x\Vert_A \ge c_1 \, \epsilon$ if $\Vert x\Vert_A \ge \epsilon$.
Therefore, for every $\epsilon>0$,
\begin{equation*}
\lim_{N\to\infty} \bb P^{\scs N}_{x_{N}} \big[\, 
\Vert X^\delta_{t_N}\Vert_A \;\ge\; \epsilon \,\big] 
\;\le\;
\limsup_{N\to\infty} \bb P^{\scs N}_{x_{N}} \big[\, 
F_\epsilon( X^\delta_{t_N}) \;\ge\; c_1 \, \epsilon \,\big] \;.  
\end{equation*}
The right hand side vanishes in view of Lemma \ref{le3} and because
$F_\epsilon(x)=0$.
\end{proof}

Recall the linear map $\Upsilon: \Sigma \to \Sigma_B$ defined in
\eqref{bsu}. Denote by $\phi_j:\Sigma_B\to \bb R$ the coordinate map
$\phi_j(x)=x_j$, $j\in B$. By Lemma \ref{s04}, the function
$\phi_j\circ\Upsilon$ belongs to $\mc D_A$ for all $j\in B$ and so we
may apply Lemma \ref{le3} for each $F=\phi_j\circ\Upsilon$.

\begin{corollary}
\label{le2}
Under the assumptions of Lemma \ref{le3}, for any small enough
$\epsilon>0$,
\begin{equation*}
\lim_{N\to\infty} \bb P^{\scs N}_{x_{N}} \big[\, 
\Vert X^{\delta}_{t_N} - x\Vert_{B} \,\ge\, \epsilon\big] 
\;=\; 0 \;.  
\end{equation*}
\end{corollary}

\begin{proof}
Since $\Upsilon(x)=x_B$, it is easy to verify that
\begin{equation*}
\Vert X^{\delta}_{t_N} - x\Vert_{B} \;\le\; 
\Vert \Upsilon(X^{\delta}_{t_N}) - \Upsilon(x) \Vert_B 
\;+\; C_0 \Vert X^{\delta}_{t_N} \Vert_{A}\;,
\end{equation*}
for some finite constant $C_0>0$. The assertion of the lemma follows
therefore by applying Lemma \ref{le3} to each function $\phi_j\circ
\Upsilon\in \mc D_A$, and by Corollary \ref{le1}.
\end{proof}

\begin{proposition}
\label{proptight}
For any sequence $x_N\in \Sigma_N$, $N\ge 1$, the sequence of laws
$\{\bb P^{\scs N}_{x_N} : N\ge 1\}$ is tight. Moreover, every limit
point of the sequence is concentrated on continuous trajectories.
\end{proposition}

\begin{proof}
It is enough to prove that the conditions of Lemma \ref{prolt} are in
force. To keep notation simple, we show that the conditions are
fulfilled for a sequence $t_N\to 0$, $x_N\to x\in \Sigma$, $x_N\in
\Sigma_N$. Let $A=\ms A(x)$, $B=\ms B(x)$, and let $\delta>0$ be
such that $\min_{j\in B} x_j\ge 2\delta$.

Recall the definition of the stopped process $X^\delta_t$ introduced
just before Lemma \ref{le3}. To prove that
\begin{equation*}
\lim_{N\to\infty} \bb P_{x_{N}}^{\scs N} \big[ \, \Vert X_{t_{N}} 
- x_{N} \Vert \ge \epsilon \, \big] \;=\; 0\;.
\end{equation*}
for $\epsilon<\delta$, it is enough to show that
\begin{equation*}
\lim_{N\to\infty} \bb P_{x_{N}}^{\scs N} \big[ \, \Vert X^\delta_{t_{N}} 
- x \Vert \ge \epsilon \, \big] \;=\; 0\;.
\end{equation*}

Fix $0<\epsilon<\delta$. Clearly, for $N$ large enough,
\begin{equation*}
\bb P_{x_{N}}^{\scs N} \big[ \, \Vert X^\delta_{t_{N}} 
- x \Vert \ge \epsilon \, \big] \;\le \; 
\bb P_{x_{N}}^{\scs N} \big[ \, \Vert X^\delta_{t_{N}} 
\Vert_A \ge  \epsilon/ 2 \, \big] \;+\; 
\bb P_{x_{N}}^{\scs N} \big[ \, \Vert X^\delta_{t_{N}} 
- x \Vert_{B} \ge \epsilon/2 \, \big]\;.
\end{equation*}
To complete the proof of the first assertion of the proposition, it
remains to apply Corollaries \ref{le1} and \ref{le2}. 

Any limit point of the sequence $\bb P^N_{x_N}$ is concentrated on
continuous trajectories because for any $T>0$, $\sup_{0\le t\le T}
\Vert X^N_t - X^N_{t-}\Vert\le 2/N$. Moreover, it follows from the
tightness of the sequence $\bb P^N_{x_N}$ that for every $\epsilon>0$
and every sequence $x_N$,
\begin{equation}
\label{06}
\lim_{\delta\to 0} \limsup_{N\to\infty} \bb P^N_{x_N} 
\big[ \sup_{0\le t\le \delta} \Vert X_t - X_0
\Vert \ge \epsilon \big] \;=\; 0\;.
\end{equation}
\end{proof}

\subsection{Characterization of limit points}

We prove in this subsection that any limit point of a sequence $\bb
P^N_{x_N}$ solves the martingale problem \eqref{f06}.

\begin{proposition}
\label{mp1}
Let $x_N\in \Sigma_N$, $N\ge 1$, be a sequence converging to some
$x\in \Sigma$, and denote by $\tilde{\bb P}$ a limit point of the
sequence $\bb P^{\scs N}_{x_N}$. Under $\tilde{\bb P}$, for any $H\in
\mc D_S$,
\begin{equation*}
H(X_{t}) - H(X_{0}) - \int_{0}^{t} (\mf L H)(X_s) \, ds
\end{equation*}
is a martingale.
\end{proposition}

The following replacement lemma is the key point in the proof of
Proposition \ref{mp1}. It permits to replace the functions $g_\ell (N
X_s(\ell))$ by the constants $m_\ell$, $\ell\in S$.

\begin{lemma}
\label{replace}
For any $\ell\in S$,
\begin{equation}
\label{07}
\lim_{N\to\infty} \max_{x\in \Sigma_N} \bb E^N_{x} \Big[ \; \Big( N \int_0^{1/N} 
\{ g_\ell (N X_s(\ell)) - m_\ell \} \, ds \Big)^2 \;\Big] \;=\; 0\;. 
\end{equation}
\end{lemma}

\begin{proof}
Let $x_N$ be a sequence such that
\begin{equation*}
\begin{split}
& \lim_{N\to\infty} \max_{x\in \Sigma_N} \bb E^N_{x} \Big[ \; \Big( N \int_0^{1/N} 
\{ g_\ell (N X_s(\ell)) - m_\ell \} \, ds \Big)^2 \;\Big] \\
&\quad \;=\;
\lim_{N\to\infty} \bb E^N_{x_N} \Big[ \; \Big( N \int_0^{1/N} 
\{ g_\ell (N X_s(\ell)) - m_\ell \} \, ds \Big)^2 \;\Big] \;.
\end{split}
\end{equation*}
Assume without loss of generality that $x_N$ converges to some $x\in
\Sigma$. Fix $j\in S$, and suppose first that $x(j)>0$. In this case,
since $\lim_n g_j(n)= m_j$, the assertion of the lemma for $\ell=j$
follows from \eqref{06}.  If $x(j)=0$, there exists $k\not =j$ such
that $x(k)>0$. By the previous observation, \eqref{07} holds with $k$
in place of $\ell$.

For $j,k\in S$, consider the function $u:S\to\bb R$ defined by
\begin{equation*}
u(j)\,=\, 1\;, \quad u(k)\;=\; 0\quad {\rm and} 
\quad (\mc L u)(i) \;=\; 0 \,\;{\rm for} 
\;\, i\;\not=\; j\;, k\;,
\end{equation*}
let $F(x) = u\cdot x$, and let $(M_t)_{t\ge 0}$ be the Dynkin's
martingale associated to $F$:
\begin{equation*}
M_t \;:=\; F(X_{t}) \;-\;  F(X_{0})
\;-\; N \int_0^{t} \sum_{i\in S} g_i (N X_s(i)) 
(\mc L f) (i) \, ds \;, \quad t\ge 0 \;.
\end{equation*}
On the one hand, an elementary computation shows that
\begin{equation*}
\bb E^N_{x_N} \big[ M^2_t \, \big] \;=\; \bb E^N_{x_N} \Big[ \int_0^{t} 
\sum_{i,\ell \in S} g_i(N X_{s}(i)) \, r(i,\ell) \,
[u(\ell)-u(i)]^2 \, ds \Big]\;,
\end{equation*}
so that $\lim_N \bb E^N_{x_N} [ M^2_{1/N} \,]=0$. On the other hand,
by definition of $F$ and by \eqref{06},
\begin{equation*}
\lim_{N\to\infty} \bb E^N_{x_N} \big [ \{ F(X_{1/N}) - F(X_{0}) \}^2 \big]
\;=\; 0\; .
\end{equation*}
Therefore,
\begin{equation*}
\lim_{N\to\infty} \bb E^N_{x_N} \Big[ \Big( N \int_0^{1/N} 
\sum_{i\in S} g_i (N X_s(i)) \, (\mc L u) (i) \, ds\Big)^2 
\Big] \;=\; 0\;.
\end{equation*}
As $\sum_{i\in S} m_i \, (\mc L u) (i)=0$, we may substitute in the
previous equation $g_i (N X_s(i))$ by $g_i (N X_s(i)) - m_i$.  Since
$(\mc L u)(i)=0$ for $i\not = j$, $k$, since $(\mc L u)(j) \not =0$,
and since \eqref{07} holds for $\ell=k$,
\begin{equation*}
\lim_{N\to\infty} \bb E^N_{x_N} \Big[ \Big( N \int_0^{1/N} 
\{ g_j (N X_s(j)) - m_j\}  \, ds\Big)^2 
\Big] \;=\; 0\;,
\end{equation*}
which completes the proof of the lemma.
\end{proof}

\begin{corollary}
\label{prore}
For any $t>0$, $j\in S$, and any continuous function $H:\Sigma\to \bb
R$,
\begin{equation*}
\lim_{N\to\infty} \max_{x\in \Sigma_N} \bb E^N_{x} \Big[ \, \Big| \int_0^{t} 
\{ g_j (N X_s(j)) - m_j\}  \, H(X_s) \, ds \,\Big|\, \Big] \;=\; 0\;.
\end{equation*}
\end{corollary}

\begin{proof}
Fix a sequence $x_N\in \Sigma_N$, $N\ge 1$ some $t>0$,
$j\in S$, and a continuous function $H:\Sigma\to \bb R$. Clearly, 
\begin{equation*}
\begin{split}
& \bb E^N_{x_N} \Big[ \, \Big|  \int_0^{t} \{ g_j (N X_s(j)) - m_j\}
\, H(X_s) \, ds \,\Big|\,\Big] \\
&\quad \;\le \; \sum_{k=0}^{[tN]} \bb E^N_{x_N} \Big[ \, \Big| 
\int_{k/N}^{(k+1)/N} \{ g_j (N X_s(j)) - m_j\}  \, H(X_s) \, ds
\,\Big|\, \Big] \;+\; O(\frac 1N) \;,
\end{split}
\end{equation*}
where $[a]$ stands for the integer part of $a\in \bb R$. By the Markov
property, the first term on the right hand side is bounded by
\begin{equation*}
[tN] \max_{x\in\Sigma_N} \bb E^N_{x} \Big[ \, \Big|
\int_{0}^{1/N} \{ g_j (N X_s(j)) - m_j\}  \, H(X_s) \, ds
\,\Big|\, \Big]\;.
\end{equation*}
Since $H$ is a continuous function, for every $\delta>0$, there exists
$\epsilon>0$ such that the previous expression is bounded by
\begin{equation*}
\begin{split}
& \delta \;+\; C_0 \max_{x\in\Sigma_N} \bb P^N_{x} 
\big[ \sup_{0\le s\le 1/N} \Vert X_s
- x \Vert \ge \epsilon \big] \\
&\quad \;+\; [tN] \max_{x\in\Sigma} |H(x) | \, 
\max_{x\in\Sigma_N} \bb E^N_{x} \Big[
\, \Big| \int_{0}^{1/N} \{ g_j (N X_s(j)) - m_j\}  \, ds
\,\Big|\, \Big]
\end{split}
\end{equation*}
for some finite constant $C_0$ which depends on $H$ and $t$. By
\eqref{06}, the second term vanishes as $N\uparrow\infty$. The third
one vanishes by Lemma \ref{replace}.
\end{proof}

\begin{proof}[Proof of Proposition \ref{mp1}]
Fix a function $H$ in $\mc D_S$, $n\ge 1$, a continuous function $G:
\Sigma^n \to \bb R$, and $0\le s_1 \le \cdots \le s_n \le t_1<t_2$. Define
\begin{equation*}
G(X) \;:=\; G(X_{s_i} , \dots, X_{s_n}) \quad {\rm and} \quad  
\Psi_{t_1,t_2} \;:=\; H(X_{t_2}) - H(X_{t_1}) - \int_{t_1}^{t_2} (\mf L
H)(X_r) \, dr \;. 
\end{equation*}
Fix a sequence $x_N\in\Sigma_N$, $N\ge 1$ and let $\tilde{\bb P}$
be a limit point of the sequence $\bb P^{\scs N}_{x_N}$. Assume,
without loss of generality, that $\bb P^{\scs N}_{x_N}$ converges to
$\tilde{\bb P}$ in the Skorohod topology. For each $N\ge 1$,
\begin{equation*}
\bb E^{\scs N}_{x_N} \Big[ G(X) \, \Big\{ H(X_{t_2}) - H(X_{t_1}) -
\int_{t_1}^{t_2} (\mf L_N H)(X_r) \, dr\Big\} \, \Big] \;=\; 0\;.
\end{equation*}
By Lemma \ref{texp}, this left hand side is equal to 
\begin{equation*}
\bb E^{\scs N}_{x_N} \big[ G(X) \, \Psi_{t_1,t_2} \, \big] \;+\;
\frac{1}{2} \sum_{j\in S} \bb E^{\scs N}_{x_N} \Big[ G(X)
\int_{t_1}^{t_2} \{ g_j (N X_s(j)) - m_j \} (\Delta_j H) (X_s) \, ds \Big] 
\end{equation*}
plus a remainder which vanishes as $N\uparrow\infty$.  By the Markov
property and by Corollary \ref{prore}, the second term vanishes as
$N\uparrow\infty$. Since $\bb P^{\scs N}_{x_N}$ converges in the
Skorohod topology to $\tilde{\bb P}$ and since the measure $\tilde{\bb
  P}$ is concentrated on continuous paths, the first term converges to
\begin{equation*}
\tilde{\bb E} \Big[ G(X) \, \Big\{ H(X_{t_2}) - H(X_{t_1}) -
\int_{t_1}^{t_2} (\mf L H)(X_r) \, dr\Big\} \, \Big] \;.
\end{equation*}
Putting together the previous estimates we conclude that this latter
expectation vanishes. This completes the proof of the proposition.
\end{proof}

For each $x\in \Sigma$, consider a sequence $x_N\in \Sigma_N$, $N\ge
1$ converging to $x$ and a limit point $\tilde{\bb P}_x$ for the
sequence $\bb P^N_{x_N}$, $N\ge 1$. Since $\tilde{\bb P}_x$ is
concentrated on $C(\bb R_+,\Sigma)$ the restriction of $\tilde{\bb
  P}_x$ to this space turns out to be a probability measure starting
at $x$. By Proposition \ref{mp1}, such restriction is a solution of
the $\mf L$-martingale problem starting at $x$. We have thus proved
the existence of solutions for the $\mf L$-martingale problem and the
proof of Theorem \ref{r03} is concluded. On the other hand, Theorem
\ref{mt2} is an immediate consequence of Proposition \ref{mp1} and of
the uniqueness of the $\mf L$-martingale problem established in the
last section.

\subsection{Additional Properties}
\label{adpro}
In this subsection we prove some additional properties of the
solution $\{\bb P_x : x\in \Sigma\}$. We start showing Feller
continuity.

\begin{proposition}
\label{fcont}
Let $(x_n)_{n\ge 1}$ be a sequence in $\Sigma$ converging to some
$x\in \Sigma$. Then $\bb P_{x_n} \to \bb P_{x}$ in the sense of weak
convergence of measures on $C(\bb R_+,\Sigma)$.
\end{proposition}

\begin{proof}
For each $z\in \Sigma$, let $\tilde{\bb P}_z$ stand for the
probability measure on $D(\bb R_+,\Sigma)$ induced by $\bb P_z$ and
the inclusion of $C(\bb R_+,\Sigma)$ into $D(\bb R_+,\Sigma)$, and
denote by $\tilde{\bb E}_{z}$ the respective expectation. 

Fix a bounded, continuous function $\Gamma:D(\bb R_+,\Sigma)\to \bb
R$. By Theorem \ref{mt2}, there exists a strictly increasing sequence
$N(n)\in \bb N$, $n\ge 1$, and a sequence $(z_n)_{n\ge 1}$ so that
$z_n\in \Sigma_{N(n)}$,
\begin{equation}
\label{cc1}
\Vert z_n - x_n\Vert \;<\; \frac{1}{n} \quad\textrm{and}\quad 
\Big| \,\bb E^{N(n)}_{z_n}[ \,\Gamma\, ] \,-\,  \tilde{\bb E}_{x_n}[
\,\Gamma \,] \,\Big| \; < \; \frac{1}{n}\;,
\end{equation}
for all $n\ge 1$. In particular, $z_n\to x$. By Theorem
\ref{mt2},
\begin{equation*}
\bb E^{N(n)}_{z_n}[ \,\Gamma\, ] \;\to\; \tilde{\bb E}_{x}
[ \,\Gamma\, ]\;,\quad \textrm{as $n\uparrow\infty$}\;.
\end{equation*}

From the previous convergence and from the second assertion in
\eqref{cc1}, it follows that $ \tilde{\bb E}_{x_n}[ \,\Gamma \,] \to
\tilde{\bb E}_{x}[ \,\Gamma\, ]$, as $n\uparrow \infty$. We have thus
shown that $\tilde{\bb P}_{x_n} \to \tilde{\bb P}_x$ in the sense of
weak convergence of measures in $D(\bb R_+,\Sigma)$. Since every
$\tilde{\bb P}_z$, $z\in \Sigma$, is concentrated on $C(\bb
R_+,\Sigma)$, this implies the desired result.
\end{proof}

Next result asserts that $\{\bb P_x : x\in \Sigma\}$ satisfies the
strong Markov property.

\begin{proposition}
\label{smark}
Fix $x\in \Sigma$. Let $\tau$ be a finite stopping time and $\{\bb
P^{\tau}_{\omega}\}$ be a r.c.p.d. of $\bb P_x$ given $\ms
F_{\tau}$. Then, there exists a $\bb P_x$-null set $\ms N\in \ms
F_{\tau}$, such that
\begin{equation*}
\bb P^{\tau}_{\omega} \circ \theta_{\tau(\omega)}^{-1} \;=\; 
\bb P_{X_{\tau}(\omega)}\;, \quad \omega \in \ms N^c \;,
\end{equation*}
where we recall $(\theta_t)_{t\ge 0}$ is the semigroup of time translations.
\end{proposition}

\begin{proof}
By Theorems \ref{amp0} and \ref{amp1}, $\bb P_x$ is an
absorbing solution of the $\ms L$-martingale problem. The same
argument used in Lemma \ref{svl} shows that there exists a $\bb
P_x$-null set $\ms N\in \ms F_{\tau}$ such that for all $\omega\in \ms
N^c$ the probability $\bb P^{\tau}_{\omega}\circ
\theta^{-1}_{\tau(\omega)}$ is an absorbing solution of the $\ms
L$-martingale problem starting at $X_{\tau}(\omega)$. By the
uniqueness result in Proposition \ref{uni2} and by Theorem \ref{amp1},
we conclude that $\bb P^{\tau}_{\omega}\circ
\theta^{-1}_{\tau(\omega)}= \bb P_{X_{\tau}(\omega)}$ for all $\omega
\in \ms N^c$, which is the content of the proposition.
\end{proof}

To conclude this section we give an estimate for the expected value of
the absorbing time $\sigma_1$ uniformly on the starting point $x\in
\Sigma$.

\begin{proposition}
\label{r06}
Let $z\in \Sigma$ be such that $z\not= \bs e_j$, $j\in S$. For any $q>b$,
\begin{equation*}
\bb E_z [ \sigma_1 ] \;\le\; 
\frac {|B|^{(q-1)\vee 1}}{ (q+1) (q-b) d(B)} \;,
\end{equation*}
where $B=\ms B(z)$ and $d(B) := \min_{j\in B} \<(-\mc S) \bs e_j , \bs
e_j\>_{\bs m}$. In particular, $\bb P_x [ \sigma_1<\infty] = 1$.
\end{proposition}


\begin{proof}
Fix $q > b$, $z\in \Sigma$, and let $B=\ms B(z)$. For $j\in B$ and
$\epsilon>0$, there exists a function $F_{\epsilon}\in D_0(\Sigma)$
such that
\begin{equation*}
F_{\epsilon}(x) \;=\; x_j^{q+1} \;,\quad x\in \Sigma_{B,0}\cap 
\Lambda_B(\epsilon/2)\;.
\end{equation*}
Recall the definition of the stopping time $h_{B}(\epsilon)$ given in
\eqref{hbep}. By Theorem \ref{amp1},
\begin{equation*}
F_\epsilon (X_{t\wedge h_B(\epsilon)})  \;-\; \int_0^{t\wedge
  h_B(\epsilon)} (\ms L F_\epsilon)(X_s)\, ds
\end{equation*}
is a $\bb P_z$-martingale, so that
\begin{equation}
\label{ezs}
\bb E_z \left[ \int_0^{t\wedge h_B(\epsilon)} 
(\ms L F_\epsilon)(X_s)\, ds \right] \;=\; 
\bb E_z[ F_\epsilon (X_{t\wedge h_B(\epsilon)}) - F(z)] \;\le \; 1 \;.
\end{equation}
By definition of $F_{\epsilon}$, for all $x\in \Sigma_{B,0}\cap
\Lambda_B(\epsilon)$ we have
\begin{equation}
\label{qqb}
\ms L F_{\epsilon}(x) \;=\; b(q+1)\sum_{k\in B}
\frac{x_j^q}{x_k} m_k \, \mc L^B \bs e_j(k) 
\;+\; q(q+1) x_j^{q-1} \< (-\mc S^B) \bs e_j, \bs e_j\>_{\bs m} \;,
\end{equation}
Since $\mc L^B \bs e_j(k)\ge 0$ for $k\not =j$ and $m_j\mc L^B \bs
e_j(j)=-D_B(\bs e_j,\bs e_j)$ then the expression in \eqref{qqb} is
bounded below by
\begin{equation*}
(q-b)(q+1) \,x_j^{q-1} D_B( \bs e_j , \bs e_j ) \;.
\end{equation*}
By using this bound in \eqref{ezs}, definition of $h_B(\epsilon)$ and
the absorbing property we get
\begin{equation*}
\bb E_z \left[ \int_0^{t\wedge h_B(\epsilon)} X_s(j)^{q-1} \, ds
\right] \;\le \; \frac{1}{(q-b)(q+1) d(B) }
\end{equation*}
Averaging over $j\in B$ in the above inequality and by using that
\begin{equation*}
\frac{1}{|B|} \sum_{j\in B} x_j^{q-1} \;\ge\; \frac{1}{|B|^{(q-1)\lor 1}}
\end{equation*}
we obtain
\begin{equation*}
\bb E_z[t\wedge h_B(\epsilon)] \;\le \; \frac{|B|^{(q-1)\lor 1}}
{(q-b)(q+1) d(B) }\cdot
\end{equation*}
It remains to let $t\uparrow\infty$ to complete the proof of the
lemma.
\end{proof}

\section{Proof of Lemma \ref{propi}}
\label{sec05}

In this section we prove Lemma \ref{propi}, which has been used in
Lemmas \ref{g1} and \ref{r01} for the construction of suitable
functions and in Corollary \ref{le1} for the proof of
tightness.\smallskip

Recall that $A$ is a proper nonempty subset of $S$. For each $j\in A$
let $\bs w_j$ represent the canonical projection of the vector $\bs
v_j$ on $\bb R^A$. Also let $\{\bs e_k : k\in A\}$ represent here the
canonical basis of $\bb R^A$. We start observing the following
relation between these two sets of vectors.

\begin{lemma}
\label{farfan}
Let $D$ and $D'$ be two different subsets of $A$ and suppose that
\begin{equation}
\label{zz1}
\sum_{j\in D} \alpha_j \bs w_j + \sum_{k\in A\setminus D} \alpha_k 
\bs e_k \;=\; \sum_{j\in D'} \beta_j \bs w_j + \sum_{k\in A\setminus D'} \beta_k \bs e_k\;,
\end{equation}
for some $\alpha_j,\beta_j\in \bb R$ such that $\beta_j\alpha_j\ge 0$
for all $j\in A$. Then $\alpha_j=\beta_j$ for all $j\in A$ and
$\alpha_k=\beta_k=0$, $\forall k\in D\Delta D'$.
\end{lemma}

\begin{proof}
Without loss of generality we may suppose that $\alpha_j\ge 0$ and
$\beta_j \ge 0$, for all $j\in A$. Otherwise, we may interchange the
respective terms in equation \eqref{zz1}. Fix an arbitrary $j_0\in
D\setminus D'$ and consider the vector $u\in \bb R^S$ defined by
$u(j_0)=1$ and
\begin{equation*}
\left\{
\begin{array}{ll}
u(k) = 0, & \textrm{for $k\in S\setminus D$}; \\
\mc L u(j)=0, & \textrm{for $j\in D\setminus \{j_0\}$}.
\end{array}
\right.
\end{equation*}
Since $u\equiv 0$ on $S\setminus D$ then the inner product of the
canonical projection of $u$ on $\bb R^A$ and the expression in the
left hand side of \eqref{zz1} equals
\begin{equation}
\label{zz3}
\Big(\sum_{j\in D} \alpha_j \bs v_j\Big) \cdot u \;=\; \sum_{j\in D} \alpha_j \mc Lu(j) \;. 
\end{equation}
Since $\mc Lu \equiv 0$ on $D\setminus \{j_0\}$, the last expression
reduces to $\alpha_{j_0}\mc Lu(j_0)$. On the other hand, the inner
product of the canonical projection of $u$ on $\bb R^A$ and the
expression in the right hand side of \eqref{zz1} is equal to
\begin{equation}
\label{zz2}
\sum_{j\in D'} \beta_j \mc Lu(j) \;+\; \sum_{k\in A\setminus D'} \beta_k u(k) \;.
\end{equation}
Since $u\in [0,1]^S$ and $u\equiv 0$ on $D'\setminus D$ then $\mc
Lu(j)\ge 0$ for all $j\in D'\setminus D$ and so, the first term in
\eqref{zz2} is positive. Therefore, \eqref{zz2} is bounded below by
\begin{equation*}
\sum_{k\in A\setminus D'} \beta_k u(k) \;=\; \sum_{k\in D\setminus D'} \beta_k u(k) \;. 
\end{equation*}
From \eqref{zz3} and the last estimate we conclude that
\begin{equation}
\label{zz4}
\alpha_{j_0}\mc Lu(j_0) \;\ge\; \sum_{k\in D\setminus D'} \beta_k u(k) \;\ge\; \beta_{j_0} \;.
\end{equation}
Since $\bs r$ is irreducible and $S\setminus D$ is not empty ($A$ is a
proper subset of $S$) then $\mc Lu(j_0)$ is strictly negative. By
using this observation in inequality \eqref{zz4} we get
$\alpha_{j_0}=\beta_{j_0}=0$. By interchanging $D$ and $D'$ in the
argument we finally conclude that
\begin{equation}\label{zz5}
\alpha_{k} \;=\; \beta_{k} \;=\; 0\;, \quad \forall k\in D\Delta D'\;.
\end{equation}
By inserting \eqref{zz5} in \eqref{zz1} we get
\begin{equation}
\label{zz6}
\sum_{j\in \tilde D} \alpha_j \bs w_j + \sum_{k\in \tilde C} 
\alpha_k \bs e_k \;=\; \sum_{j\in \tilde D} \beta_j \bs w_j 
+ \sum_{k\in \tilde C} \beta_k \bs e_k\;,
\end{equation}
where $\tilde D = D\cap D'$ and $\tilde C=A\setminus (D\cup D')$. Fix
an arbitrary $j_1\in \tilde D$ and consider now the vector $v\in \bb
R^S$ defined by $v(j_1)=1$ and
\begin{equation*}
\left\{
\begin{array}{ll}
v(j) = 0, & \textrm{for $j\in S\setminus \tilde D$}; \\
\mc L v(j)=0, & \textrm{for $j\in \tilde D\setminus \{j_1\}$}.
\end{array}
\right.
\end{equation*}
Similarly to the computation we performed for equation \eqref{zz1}, we
take the inner product of the canonical projection of $v$ on $\bb R^A$
and each term in equation \eqref{zz6} to get
\begin{equation}
\alpha_{j_1}\mc Lv(j_1) \;=\; \beta_{j_1}\mc Lv(j_1)\;.
\end{equation}
Since $S\setminus \tilde D$ is nonempty and $\bs r$ is irreducible
then $\mc Lv(j_1)<0$ implying that $\alpha_{j_1}= \beta_{j_1}$. We
have thus proved that $\alpha_{j}=\beta_{j}$ for all $j\in \tilde
D$. Therefore, from \eqref{zz6} we conclude that actually
$\alpha_{j}=\beta_{j}$ for all $j\in \tilde D \cup \tilde C$.
\end{proof}

\begin{corollary}
\label{cor71}
For any $D\subseteq A$, the set of vectors
\begin{equation*}
\{ \bs w_j : j\in D \} \cup \{\bs e_k : k\in A\setminus D\}
\end{equation*}
is a basis of $\bb R^A$.
\end{corollary}
\begin{proof}
Suppose that
\begin{equation*}
\sum_{j\in D} \alpha_{j} \bs w_j \;+\; 
\sum_{k\in A\setminus D}\alpha_k \bs e_k \;=\; 0\;.
\end{equation*}
By applying the previous lemma with $\beta_j=0$, $j\in A$ and
$D'=A\setminus D$ we get $\alpha_j=0$ for all $j\in A$ proving the
desired result.
\end{proof}

\begin{corollary}
\label{cor72}
For every $v\in \bb R^A$ there exists $D\subseteq A$ such that
\begin{equation*}
v \;=\; \sum_{j\in D} \alpha_{j} \bs w_j \;+\; \sum_{k\in A\setminus D}\alpha_k \bs e_k
\end{equation*}
with $\alpha_j\ge 0$, for all $j\in A$.
\end{corollary}

\begin{proof}
Denote by $\ms W$ the set of vectors in $\bb R^A$ for which all the
coordinates with respect to the basis
\begin{equation}
\label{basis}
\{ \bs w_j : j\in D \} \cup \{\bs e_k : k\in A\setminus D\}
\end{equation}
are non zero, for every $D\subseteq A$. Fix some $x\in \ms W$. For
each $D\subseteq A$, denote by $\sigma_D\in \{-1,+1\}^S$ the vector
\begin{equation*}
\sigma_{D}(j)\;:=\;\left\{
\begin{array}{ll}
+1 \;, & \textrm{if $\alpha_j>0$}\;, \\
-1 \;, & \textrm{if $\alpha_j>0$}\;,
\end{array}
\right.
\end{equation*}
where $\alpha_j$, $j\in A$ are the coordinates of $x$ with respect to
the basis \eqref{basis}. By Lemma \ref{farfan}, $\sigma_D \not=
\sigma_{D'}$ for $D\not=D'$. Since $\{-1,+1\}^S$ and the powerset of
$A$ have the same cardinality then there must exist some $D_0\subseteq
A$ such that $\sigma_{D_0}\equiv +1$. This shows the assertion for
every vector in $\ms W$. Since $\ms W$ is dense in $\bb R^A$ and the
set of vectors satisfying the assertion is closed in $\bb R^A$ then
the proof is complete.
\end{proof}

We will also need the following observation 
\begin{lemma}
\label{r17}
Fix some $D\subseteq A$, $j_0\in D$ and write
\begin{equation}
\label{jk1}
\bs e_{j_0} \;=\; \sum_{j\in D}
\alpha_j \bs w_j \;+\; \sum_{k\in A\setminus D} \alpha_k \bs e_k\;.
\end{equation}
We have $\alpha_k\ge 0$, $k\in A\setminus D$.
\end{lemma}

\begin{proof}
Fix an arbitrary $k\in A\setminus D$ and consider the vector $v\in \bb
R^S$ defined by $v(k)=1$, $v\equiv 0$ on $(D\cap \{k\})^c$ and $\mc
Lv\equiv 0$ on $D$. Taking the inner product of the projection of $v$
on $\bb R^A$ and each term in equation \eqref{jk1} we get
$v(j_0)=\alpha_{k}$. Since $v\in [0,1]^S$, we are done.
\end{proof}

For each subset $D$ of $A$, let $\mf C_D$ be the closed cone generated
by the vectors in \eqref{basis}:
\begin{equation*}
\mf C_D \;:=\; \Big\{\, \sum_{j\in D} \alpha_{j} \bs w_j \;+\; 
\sum_{k\in A\setminus D}\alpha_k \bs e_k \;:\; \alpha_j\ge 0\,,\; j\in A \,\Big\}\;.
\end{equation*}
By Corollary \ref{cor71}, each cone $\mf C_D$ is $|A|$-dimensional and
by Corollary \ref{cor72} we have
\begin{equation}\label{unn}
\bigcup_{D:D\subseteq A} \mf C_D \;=\; \bb R^A \;.
\end{equation}
As an immediate consequence of Lemma \ref{farfan} we have, for any two
subsets $D$, $D'$ of $A$, that
\begin{equation}
\label{inco}
\mf C_D \cap \mf C_{D'} \;=\; \Big\{ \,\sum_{\tilde D} \alpha_j \bs
w_j \;+\; \sum_{k\in \tilde C} \alpha_j \bs e_k \;:\; \alpha_j\ge
0\,,\; 
j\in \tilde D \cup \tilde C \,\Big\}
\end{equation}
where $\tilde D = D\cap D'$ and $\tilde C= A \setminus (D\cup
D')$. Note that $\mf C_\varnothing$ corresponds to the positive
quadrant, $\mf C_\varnothing = \bb R_+^{A}$. We also note
that 
\begin{equation}
\label{f30}
\mf C_A \;\subseteq\; \Big\{x\in \bb R^A : \sum_{j\in A} x_j \le 0 \Big\}\;.
\end{equation}
Indeed, if $x=\sum_{j\in A} \alpha_j \bs w_j$ with $\alpha_j\ge 0$ for
all $j\in A$ and ${\mb 1}_A\in \bb R^S$ is the indicator of
$A\subsetneq S$ then
\begin{equation*}
\sum_{j\in A} x_j \;=\; \Big(\sum_{j\in A} \alpha_j \bs v_j \Big) 
\cdot {\mb 1}_A \;=\; \sum_{j\in A} \alpha_j \mc L{\mb 1}_A(j) \;\le \; 0
\end{equation*}
because $\mc L {\mb 1}_A(j)\le 0$ for all $j\in A$. 

We finally need the following observation about the position of the cones.
\begin{lemma}
\label{f29}
We have
\begin{equation*}
\{x\in \bb R^A : x_k\le 0 \} \;\subseteq\; \bigcup_{D: k\in D} \mf C_D\;.
\end{equation*}
for any $k\in A$.
\end{lemma}

\begin{proof}
Let $x\in \bb R^A$ be such that $x_k < 0$. By \eqref{unn}, $x\in \mf
C_D$ for some $D\subseteq A$. Since among the vectors $\{\bs w_j :
j\in B\}$, $\{\bs e_j : j\in B\}$ only the vector $\bs w_k$ has its
$k$-th coordinate negative, $D$ must contain $k$. Since the cones are
closed, the assertion follows from this observation.
\end{proof}

We may now conclude the proof of Lemma \ref{propi}. For each
$\epsilon>0$ denote $Q_\epsilon = [-\epsilon , \infty)^A$. Let $G:\bb
R_+^A\to \bb R$ be the function defined by
\begin{equation*}
G(x) \;=\; \sum_{j\in A} x_j\;.
\end{equation*}
The idea of the proof is to extend $G$ to $Q_\epsilon$ in a linear way
for boundary conditions in Definition \ref{bc1} to be fulfilled in a
neighborhood of the boundary of $Q_\epsilon$. We first extend $G$ to
$\bb R^A$ as follows. Fix $x \in \bb R^A$. According to \eqref{unn},
$x\in \mf C_D$ for some $D\subseteq A$. We then define
\begin{equation*}
F (x) \;=\; G\Big( \sum_{j\in A\setminus D} \alpha_k \bs e_k \Big)\;.
\end{equation*}
where $\alpha_k\ge 0$, $k\in A$ are the coordinates of $x$ in the
basis \eqref{basis}. Observation \eqref{inco} assures that $F$ is well
defined. Since $\mf C_\varnothing = \bb R_+^A$ then $F$ and $G$
coincide on $\bb R_+^A$. Moreover, $F$ is constant along the vector
$\bs w_j$ in the cone $\mf C_D$ if $D$ contains $j$. In particular, by
Lemma \ref{f29},
\begin{equation}
\label{f24}
\partial_{\bs w_j} F (x) \;=\; 0
\quad \text{for $x$ such that }\; x_j < 0 \;,
\end{equation}
where $\partial_{\bs w_j}$ stands for the directional derivative along
$\bs w_j$. It is also clear from the definition of the function $F$
that in the interior of the cone $\mf C_D$, denoted hereafter by
$\mathring{\mf C}_D$, $F$ increases in the $\bs e_k$ direction for
$k\not\in D$. Actually, $(\partial_{x_k} F)(x) =1$ for $x\in
\mathring{\mf C}_D$, $k\not\in D$.  Moreover, in view of Lemma
\ref{r17}, in $\mathring{\mf C}_D$, $(\partial_{x_k} F)(x) \ge 0$ for
$k\in D$. Therefore, in the interior of the cone $\mf C_D$
\begin{equation}
\label{f25}
(\partial_{x_k} F)(x) \;=\; 1  \quad \text{ for } k\not\in D
\;, \quad \text{ and } \quad
(\partial_{x_j} F)(x) \;\ge\; 0 \quad \text{ for } j\in D \;.
\end{equation}

The function $F$ is clearly not $C^2$ because its partial derivatives
are not continuous. To remedy, we convolve it with a smooth
mollifier. Let $\varphi : \bb R^A \to \bb R_+$ be a mollifier:
$\varphi$ is a smooth function whose support is contained in
$[-|A|^{-1/2},|A|^{-1/2}]^A$, and $\int_{\bb R^A} \varphi(x)\,
dx=1$. For $\delta>0$, let $\varphi_\delta(x) = \delta^{-|A|}
\varphi(x/\delta)$. Fix $\delta>0$, and denote by $F_{\delta} : \bb
R^A \to \bb R$ the function obtained by taking the convolution of $F$
with $\varphi_\delta$. Clearly, the function $F_{\delta}$ is
smooth. Since the boundaries of the cones have null Lebesgue measure,
by \eqref{f24} and \eqref{f25},
\begin{equation}
\label{f31}
(\bs w_j \cdot \nabla F_{\delta}) (x) \;=\; 0\;\; \text{for }\; x_j\le
-\delta \;, \qquad \sum_{k\in A} (\partial_{x_k} F_{\delta})(x) \;\ge
\; 1 \;\; \text{if } \; d(x,\mf C_A) \ge \delta\;,
\end{equation}
where $d(x,\mf C_A)$ represents the distance from $x$ to $\mf C_A$. 

Fix $\epsilon\ge \max\{ 2, |A|^{1/2}\} \delta$, and let $\mf S^+_a =
\{x\in \bb R^A : \sum_{j\in A} x_j \ge a\}$, $\mf S^-_b = \{x\in \bb
R^A : \sum_{j\in A} x_j \le b\}$. Since $d(\mf S^+_\epsilon, \mf
S^-_0) = \epsilon/|A|^{1/2} \ge \delta$ and since, by \eqref{f30},
$\mf C_A \subseteq \mf S^-_0$, the second property in \eqref{f31} holds
in $\mf S^+_\epsilon$.

Recall that $Q_\epsilon = [-\epsilon , \infty)^A$, and let $A = \max\{
F_{\delta}(x) : x\in Q_\epsilon \cap \mf S^-_ \epsilon\}$. The
constant $A$ is finite because $F_{\delta}$ is a continuous function
and $Q_\epsilon \cap \mf S^-_ \epsilon$ is a compact set. Fix $a>A$,
and denote by $\mf M \subset Q_{\epsilon}$ the $a$-level set of
$F_{\delta}$ in $Q_\epsilon$: $\mf M = \{x\in Q_{\epsilon}: F_{\delta}
(x)=a\}$. By the choice of $a$, $\mf M$ is contained in the interior
of $\mf S^+_\epsilon$.  Let $\mf S_+$ be the intersection of the
radius-one sphere with $\bb R_+^A$: $\mf S_+ = \{x\in \bb R_+^A : \Vert x\Vert
=1\}$. Let $\bs \epsilon$ be the vector $(-\epsilon, \dots,
-\epsilon)$. For each $x\in \mf S_+$, there exists a unique $r>0$ such
that $\bs \epsilon + r x\in \mf M$, that is, such that $F_{\delta}
(\bs \epsilon + r x)=1$. Existence follows from the continuity of
$F_{\delta}$ and from the fact that $F_{\delta} (\bs \epsilon)=0$,
$\lim_{r\to\infty} F_{\delta} (\bs \epsilon + r x)=\infty$. The point
$r$ is unique because $F_{\delta} (\bs \epsilon + r x)$ is strictly
increasing in $r$ in the set $\mf S^+_\epsilon$ in view of the second
property in \eqref{f31}, and because the level set $\mf M$ is
contained in the interior of $\mf S^+_\epsilon$ by definition of $a$.

We are finally in a position to define the function $I_A$. Let $J_A:
\Sigma \to \bb R_+$ be given by $J_A(x)=0$ if $x_A=0$ and, otherwise,
\begin{equation*}
J_A(x) \;=\; s\;, 
\end{equation*}
where $s$ is the unique $r>0$ such that $\bs \epsilon + x_A/r\in \mf
M$.  Note that the canonical projection of the level set $\{x\in
\Sigma : J_A(x) = 1\}$ on $\bb R^A$ corresponds to the set $-\bs
\epsilon + \mf M$. It is clear from the definition of $J_A$ that there
exists finite constants $0<c_1 < C_1$ such that
\begin{equation*}
c_1 \Vert x \Vert_A \;\le\; J_A(x) \;\le\;
C_1 \Vert x \Vert_A \;.
\end{equation*}
We then set $I_A(x) = J_A(x)^2$, $x\in \Sigma$. It is not difficult to
check that $I_A\in C^2(\Sigma)$ because the manifold $\mf M$ is
smooth. By the first property in \eqref{f31}, for each $x\in \Sigma$,
$j\in A$ such that $x_j=0$, there exists a neighborhood of $x$ in
which the function $J_A(x)$ is constant along the $\bs
v_j$-direction. Therefore $I_A$ belongs to $\mc D_A$. This completes
the proof of Lemma \ref{propi}.  \medskip

\noindent \textbf{Acknowledgments:} The authors wish to thank
J. Farfan, T. Funaki and I. Karatzas for stimulating discussion.  The
first author acknowledges financial support from the Vicerrectorado de
investigaci\'on de la PUCP.

\end{document}